\documentclass[11pt]{amsart}
\usepackage[latin1]{inputenc} 
\usepackage[pctex32]{graphics}
\usepackage{amsfonts}
\usepackage{amssymb,amscd,latexsym}
\usepackage{amsmath}
\usepackage{epsfig}
\usepackage{color}

\newtheorem{thm}{Theorem}[section]
\newtheorem{cor}[thm]{Corollary}
\newtheorem{prop}[thm]{Proposition}

\newtheorem{defin}[thm]{Definition}

\newtheorem{lema}[thm]{Lemma}
\newtheorem{rmk}[thm]{Remark}

\newtheorem{ex}{Example}

\newcommand{\codim}{\operatorname{codim}}
\newcommand{\red}{\operatorname{red}}
\newcommand{\reg}{\operatorname{reg}}
\renewcommand{\Im}{\operatorname{Im}}
\def\p{\mathbb P}

\def\P{\mathbb{P}}
\def\rk{\operatorname{rk}}
\newcommand{\dy}{\displaystyle}
\def\Hess{\operatorname{Hess}}
\def\hess{\operatorname{hess}}
\renewcommand{\Vert}{\operatorname{Vert}}
\def\Sing{\operatorname{Sing}}

\newcommand{\mult}{\operatorname{mult}}

\newcommand{\Bs}{\operatorname{Bs}}

\begin{document}

\title{On  Cubic Hypersurfaces with vanishing Hessian}
\author[R. Gondim]{Rodrigo Gondim}
\address{Universidade Federal Rural de Pernambuco}
\email{rodrigo.gondim.neves@gmail.com}
\author[F. Russo]{Francesco Russo*}
\address{Dipartimento di Matematica e Informatica, Universit\` a degli Studi di Catania, Viale A. Doria 5, 95125 Catania, Italy}
\email{frusso@dmi.unict.it}
\date{}

\begin{abstract}
We prove that for $N\leq 6$ an irreducible cubic hypersurface with vanishing hessian in $\p^N$  is either
a cone or a scroll in linear spaces tangent to the dual of the image of the polar map of the hypersurface.
We also provide canonical forms and a projective characterization of {\it Special Perazzo Cubic Hypersurfaces},
which, a posteriori, exhaust the class of cubic hypersurfaces with vanishing hessian, not cones,  for $N\leq 6$. 
Finally we show by pertinent examples the technical difficulties arising for $N\geq 7$.

\end{abstract}
\thanks{*Partially  supported  by the PRIN ``Geometria delle variet\`{a} algebriche"; the author is a member of the G.N.S.A.G.A}

\maketitle

\section*{Introduction}

The aim of this  paper is to provide the classification of cubic hypersurfaces with vanishing hessian in $\p^N$
for $N\leq 6$ (see Theorems  \ref{P4}, \ref{P5}, \ref{P6_1}, \ref{P6_1_1}, \ref{P6_2}).

Hypersurfaces with vanishing hessian 
were studied systematically  for the first time in  the fundamental paper \cite{GN}, where  Gordan and M. Noether  analyze
Hesse's claims in \cite{Hesse1, Hesse2}
according to which these hypersurfaces are necessarily cones.

More explicitly, if  $X = V(f) \subset \P^N$ is  a reduced  complex hypersurface,  
the hessian of $f$ (or by abusing the terminology the hessian of $X$), indicated by $\hess_X$,  is the determinant of the 
matrix of the second derivatives of the form $f$, that is  the determinant of the so called hessian matrix of $f$.

Of course cones have vanishing hessian and  Hesse  claimed twice in \cite{Hesse1} and in  \cite{Hesse2} that a 
hypersurface $X\subset\p^N$ is a cone if  $\operatorname{hess}_X= 0$. Clearly the claim is true if $\deg(X)=2$ so that the first relevant case for the problem is that of cubic hypersurfaces. One immediately sees
that $V(x_0x_3^2 +  x_1x_3x_4 +  x_2x_4^2)\subset \p^4$ is a cubic hypersurface with vanishing hessian but
not a cone, for example because the first partial derivatives of the equation are linearly independent.

Actually the  question is quite subtle
because,  as it was firstly pointed out in \cite{GN},
the claim is true for $N\leq 3$ and in general false for $N\geq 4$. The cases $N=1,2$ are easily handled but beginning
from $N=3$ the question is related to non trivial characterizations of cones among developable hypersurfaces or, from a
 differential  point of view, to a characterization of algebraic cones (or of algebraic cylinders in the affine setting)
among hypersurfaces with zero gaussian  curvature at every regular point. 
 
Gordan and Noether approach 
to the problem and their proofs for $N\leq 3$ have been revisited recently in modern terms in \cite{Lo} and \cite{GR}, see also \cite{Pt1,Pt2}.  
In \cite{GN} it is  constructed  a series of projective hypersurfaces in $\p^N$ for every $N\geq 4$, which generically are not cones  and 
to which the explicit example recalled above belongs, see also \cite{Pt1,Pt2},  \cite[Section 2]{CRS} and Example \ref{Segreproj}. Moreover Gordan and Noether 
also classified all hypersurfaces with vanishing hessian for $N\leq 4$ proving that they are either cones or $N=4$ and 
the hypersurfaces belong to their series of examples, see {\it loc. cit.}, \cite{Fr2} and \cite{GR}. The work of Perazzo considered
the classification of cubic hypersurfaces with vanishing hessian for $N\leq 6$, see \cite{Pe}, and analyzed  particular subclasses
dubbed by us {\it Special Perazzo Cubic Hypersurfaces}. Perazzo's  ideas and techniques  are very interesting and inspired deeply
our work although in our opinion they  contain some serious mistakes and non motivated claims which unfortunately
affect most of the main results,  see Remark \ref{rm:linearfiber} for a discussion of some of these imprecise statements.

As far as we know no explicit classification result is known for $N\geq 5$ so that our contribution is the first 
attempt to extend to higher dimension   Gordan-Noether-Franchetta classification of hypersurfaces with vanishing hessian
albeit only in the cubic case.

Hypersurfaces with vanishing hessian remained  outside the mainstream of (algebraic) geometry for a long time although they represent very interesting objects for many areas of research. For example the cubic hypersurface recalled above, very well known to classical algebraic geometers, is  celebrated in the modern algebraic-differential geometry literature as the 
{\it Bourgain-Sacksteder Hypersurface} (see \cite{AG2, AG, FP}). 
Moreover the regular points of a hypersurface with  
vanishing hessian are all {\it parabolic} and represent a natural generalization of the flex points of plane curves (see \cite{Ciliberto, CRS}). The divisibility properties of the hessian with respect to the original form  
have interesting geometric consequences, see \cite{Segre, Ciliberto}; for instance  {\it developable} hypersurfaces are those for which $\operatorname{hess}_X= 0 \pmod{f}$, see \cite{Segre, AG, FP}.  

We now describe the sections of the paper in  more detail. In Section \ref{Preliminaries} we introduce the notation, recall some well known results and define the polar map of a hypersurface. We also state without proof the Gordan-Noether Identity for hypersurfaces
with vanishing hessian and deduce some important geometrical consequences in Theorem \ref{partialdiff} and Corollary \ref{thm:polar&perazzo}. In Section \ref{Perazzomap} we introduce the Perazzo map of a hypersurface with vanishing hessian
following \cite{Pe} and then we study the geometry associated to the fibers of this map and their relation with the key variety $Z^*$,
which is the dual of the closure of the image of the polar map, see Section \ref{Preliminaries} for precise definitions. Firstly we prove that the closure of a general fiber of the Perazzo map is a linear space,
Theorem \ref{Perazzolinear}, correcting some mistakes contained in \cite{Pe}, see Remark \ref{rm:linearfiber} and Proposition \ref{estimateF}.
Then we analyze the congruence of linear spaces determined by the fibers of the Perazzo map introducing the notion of {\it Special Perazzo
Cubic Hypersurface} which corresponds to the case in which the congruence of the fibers of the Perazzo map is the family  of linear spaces passing through
a fixed (codimension one) linear space, see Proposition \ref{prop:lema1,2}. 

Special Perazzo Cubic Hypersurfaces  are somehow surprising since they are ruled by a family of linear
spaces
along which the hypersurface is not developable so that this ruling is different from the one  given by the fibers of the Gauss map. These examples and their generalizations are known in differential geometry as
{\it twisted planes}, see for example \cite{FP}. Despite the huge number of papers dedicated to this subject in differential geometry very few classification or structure results have been obtained. In our opinion  the global point of view provided by polarity, which  has been  overlooked until now by differential geometers, is a stronger tool to treat these objects.

In  Section \ref{examplesmall} we exhibit  canonical forms of cubics with vanishing hessian and singular along a linear space and we also prove a projective characterization  of Special Perazzo Cubic Hypersurfaces following \cite{Pe}. Section \ref{Nleq6} contains the classification and structure of cubic hypersurfaces with vanishing hessian having $\dim(Z^*)=1$,
Theorem \ref{mu = 1}. With these
results the classification of cubic hypersufaces with vanishing hessian in $\p^N$, $N\leq 6$,  is easily deduced in the rest of Section \ref{Nleq6}.
 In Section \ref{N7}  we discuss some examples showing the difficulties arising for $N\geq 7$ and the existence of some exceptional and sporadic phenomena.

\section{Preliminaries}\label{Preliminaries}

\subsection{Cones, Hessian and the  Polar Map of a Hypersurface}

Let $X = V(f) \subset \P^N$ be a reduced  hypersurface and let $d=\deg(X)\geq 1$ be its degree.
From now on we shall also restrict to the case of 
an algebraically closed field $\mathbb K$
of characteristic zero.

For a set $S\subseteq \p^N$ we shall indicate by $<S>$ its linear span
in $\p^N$ and $S\subseteq\p^N$ is said to be {\it degenerate} if $<S>\subsetneq\p^N$. With this notation  $<p,q>$ is the line through two distinct points $p,q\in\p^N$.
\begin{defin}
  Let $X \subset \P^N$ be a projective variety. The vertex of $X$ is 
$$\operatorname{Vert}(X) = \{ p \in X | <p,q> \subset X,\  \forall q \in X\}.$$

A projective variety $X \subset \P^N$ is a cone if $\operatorname{Vert}(X) \neq \emptyset$.

\end{defin}

\begin{rmk}\label{rm:lincones}{\rm

The set  $\Vert(X)\subset\p^N$ is a linear subspace and,
by generic smoothness,   $\Vert(X)=\bigcap_{x\in X}T_xX$, where $T_xX$ is the projective tangent space to $X$ at $x$.

Let $X\subset\p^N$ be an equidimensional variety of dimension $n=\dim(X)\geq 1$. Then
$\dim(\Vert(X))\geq n-1$ implies that $X$ is the union of linear spaces passing through
$\Vert(X)$.

}
\end{rmk}

We present, without proof, some  well known results used in the sequel.

\begin{prop}{\label{prop:equivalence_definitions_of_cones}}
Let $X = V(f) \subset \P^N$ be a  hypersurface of degree $d$. Then the following conditions are equivalent:
\begin{enumerate}
  \item[i)] $X$ is a cone;
\item[ii)] There exists a point $p \in X$ of multiplicity $d$;
  \item[iii)] The partial derivatives $\frac{\partial f}{\partial x_0},\frac{\partial f}{\partial x_1},...,\frac{\partial f}{\partial x_N}$ of $f$ are linearly dependent;
  \item[iv)] Up to a projective transformation, $f$ depends on at most $N$ variables.
\item[v)] The dual variety of $X$, $X^* \subset (\P^N)^*$, is degenerate.
  \end{enumerate}

\end{prop}

\begin{defin} \rm
 Let $X = V(f) \subset \P^N$ be a reduced hypersurface. The Hessian matrix of $f$ is
$$\operatorname{Hess}_f = \left[\frac{\partial^2 f}{\partial x_i \partial x_j}\right]_{0\leq i,j \leq N}$$
We also call it the hessian matrix of $X$ and write $\operatorname{Hess}_X$ since we will be interested in properties
of this matrix (like the vanishing of its determinant or more generally its rank) which are well defined modulo
the multiplication of $f$ by a non-zero constant. 
The determinant of the matrix $\Hess_X$  will be denoted by $\hess_X$
and called the {\it hessian of $X$}. Thus it is defined modulo a non zero constant and it is a projective covariant
whose vanishing does not depend on the equation of $X$. 
\end{defin}

Cones form a trivial class of hypersurfaces with vanishing hessian. Indeed if $X$ is a cone up to a linear change of coordinates the form $f$ 
does not depend on all the variables, see  also
Proposition \ref{prop:equivalence_definitions_of_cones}. Therefore in this case the Hessian matrix has at least a null row (and a null column), yielding $\hess_X= 0$.

The converse is not true in general if $d\geq 3$ as shown by the example recalled in the Introduction. We now define the polar map of a hypersurface in order to begin to clarify the deep geometrical consequences  of the condition $\hess_X= 0$.

\begin{defin}\rm
The {\it polar map} (or {\it gradient map}) of a hypersurface $X=V(f) \subset \P^N$, indicated by $\Phi_X$ or by $\Phi_f$, is the rational map  $\Phi_X:  \P^N  \dashrightarrow   (\P^N)^*$ given by the derivatives of $f$:
$$ \Phi_X(p)=(\frac{\partial f}{\partial x_0}(p): \frac{\partial
f}{\partial x_1}(p):\ldots :\frac{\partial f}{\partial x_N}(p) ).$$

Let  $Z = \overline{\Phi_X(\P^N)}\subseteq(\p^N)^*$ be the closure of the image of the polar map.
The base locus scheme of the polar map is the singular scheme of $X$ which will be denoted by $$\operatorname{Sing}X:=V(\frac{\partial f}{\partial x_0},\ldots, \frac{\partial f}{\partial x_N})\subset \p^N.$$ We shall maintain it distinct from the 
set theoretic singular locus $(\operatorname{Sing} X)_{\operatorname{red}}$.
\end{defin}

If $p=[\mathbf v]$, then  $t_p\p^N=\mathbb K^{N+1}/<\mathbf v>$ is  the affine tangent space to $\p^N$ at $p$. Let
$\Hess_X(p)$ be the equivalence class of the  hessian matrix of $X$ evaluated at $\mathbf v$.
By Euler's formula the equivalence class of  the linear map $\Hess_X(p)$  passes to the quotients and it induces the differential
of the map $\Phi_X$ at $p$:
\begin{equation}\label{eq:dfp}
(d\Phi_X)_p: t_p\p^N\to t_{\Phi_X(p)}\p^N,
\end{equation}
whose image is  $t_{\Phi_X(p)}Z$ if $p$ is general by generic smoothness. From this we can
describe the projective tangent space to $Z$ at $\Phi_X(p)$ for $p\in\p^N$ general, obtaining 
\begin{equation}\label{eq:TpZ}
T_{\Phi_X(p)}Z=\p(\Im(\Hess_X(p)))\subseteq(\p^{N})^*.
\end{equation}
Thus
\begin{equation}\label{eq:dimZdiff}
\dim Z = \operatorname{rk}(\operatorname{Hess}_X)-1.
\end{equation}
Therefore  $\hess_X= 0$  if and only if $Z\subsetneq\p^{N*}$ if and only if the partial
derivatives of $f$ are algebraically dependent, that is there exists a non zero homogeneous polynomial $g\in\mathbb K[y_0,\ldots, y_N]$
such that $g(\frac{\partial f}{\partial x_0},\ldots,\frac{\partial f}{\partial x_N})= 0$. 

From this perspective  Hesse's claim can be  translated  into asking if  the algebraic dependence of the first partial derivatives
of a homogeneous form implies their linear dependence, a quite subtle question as we mentioned in the Introduction and whose answer is negative in general but positive for $N\leq 3$, see {\it loc. cit.}.
\medskip

The next result will be important to describe the structure of hypersurfaces with vanishing hessian via the analysis of the restriction of the polar map  to a
hyperplane, for a proof see \cite[Lemma 3.10]{CRS}.

\begin{lema}\label{projection}
Let $X=V(f)\subset\mathbb{P}^N$ be a hypersurface. Let
$H=\mathbb{P}^{N-1}$ be a hyperplane not contained in $X$, let
$h=H^*$ be the corresponding point in $\mathbb{P}^{n*}$ and let
$\pi_h$ denote the projection from the point $h$. Then:
$$\Phi_{V(f)\cap H}=\pi_h\circ (\Phi_{f|H}).$$
In particular, $Z(V(f)\cap H)\subseteq \pi_h(Z(f))$, where  $Z(V(f)\cap H)$ denotes the closure of the image of the polar map 
$\Phi_{V(f)\cap H}:H=\p^{n-1}\dasharrow (\p^{n-1})^*$. 
\end{lema}
\medskip

\begin{defin}\label{dualGaussmap}\rm
  Let $X \subset \P^N$ be a reduced equidimensional projective variety of dimension $\dim X = n$. The {\it Gauss map of $X$}, $\mathcal G_X: X\dasharrow\mathbb G(n,N)$, is the rational map 
associating to every smooth point of $X$ its embedded projective tangent space considered as a point in  $\mathbb{G}(n,N)$, i.e. $\mathcal G_X(p)=[T_pX]\in\mathbb G(n,N)$.

We shall later consider also the {\it dual Gauss map of $X$} defined at a smooth point $x\in X$ by $\mathcal G_X^*(x)=[(T_xX)^*]\in\mathbb G(N-n-1, N)$.
\end{defin}

In the special case of hypersurfaces, $X = V(f) \subset \P^N$, the Gauss map of $X$ is the restriction of the polar map of $X$ to $X$
and  the image of the Gauss map is the dual variety $X^*$. In particular $X^* \subseteq Z$.

\subsection{The Gordan-Noether Identity and its geometrical consequences}\label{GNIDS}

From now on $f\in \mathbb K[x_0,\ldots,x_N]$ will be a homogeneous reduced
polynomial of degree $d$ such that $\hess_f= 0$, unless
otherwise stated.
\medskip

Since $\hess_f= 0$, there exist non zero homogeneous polynomials
$\pi\in \mathbb K[y_0,\ldots , y_N]$ such that $\pi(\frac{\partial f}{\partial
x_0},\ldots ,\frac{\partial f}{\partial x_N})=0$. Let $g\in
\mathbb K[y_0,\ldots , y_N]$ be  an irreducible polynomial with this property and such that
$g_i:=\frac{\partial g}{\partial y_i}(\frac{\partial f}{\partial
x_0},\ldots ,\frac{\partial f}{\partial x_N})\neq 0$ for at least
one $i\in\{0,\ldots,N\}$. Letting
$T=V(g)\subset\p^{N*}$ the previous condition is equivalent to $Z\not\subseteq\Sing T$.

\begin{defin}\label{defpsig}{\em Let $T=V(g)\subset\p^{N*}$ be an irreducible and reduced hypersurface  containing the polar image $Z(f)$, where $g(y_0,\ldots, y_N)$ is as above.
By definition of $g$  the variety $Z\not\subset\Sing T$ so that
$\psi_g=\Phi_g\circ\Phi_f\colon\mathbb{P}^N\dashrightarrow\mathbb{P}^N$
is well defined. Equivalently $\psi_g$ is the composition of $\phi_f$ with the Gauss map of $T$.
 If the polynomials
$g_i:=\frac{\partial g}{\partial y_i}(\frac{\partial f}{\partial x_0},\ldots ,\frac{\partial f}{\partial x_N})\in \mathbb K[x_0,\ldots ,
x_N]$ have a common
divisor $\rho:={\rm g.c.d.}(g_0,\ldots ,g_N)\in \mathbb K[x_0,\ldots ,
x_N]$,
set $h_i:=\frac{g_i}{\rho}\in \mathbb K[x_0,\ldots ,
x_N]$, for $i=0,\ldots ,N$.

It follows that the map $\psi_g$ is given  by:
\begin{equation}\label{psig}
\psi_g(p) =(g_0(f_0(p),\ldots ,f_N(p))\colon\ldots\colon g_N(f_0(p),\ldots ,f_N(p)))=(h_0(p)\colon\ldots\colon h_N(p)),
\end{equation}
with ${\rm
g.c.d.}(h_0,\ldots, h_N)=1$.}
\end{defin}

Set $T^*_Z:=\overline{\psi_g(\mathbb{P}^N)}$ and note that, by definition of $\psi_g$,
$T^*_Z\subset Z(f)^*.$ Therefore by taking $\alpha+1=\codim(Z)$ polynomials $g^0,\ldots g^{\alpha}$
defining locally $Z$ around a point $z\in Z_{\reg}$ we deduce that  $(T_zZ)^*\subset Z^*$. Moreover 
for a general point $r\in (T_{\Phi_X(p)}Z)^*$, $p\in\p^N$ such that $\Phi_X(p)\in Z_{\reg}$, there exists $a_0,
\ldots, a_{\alpha}\in \mathbb K$ such that, letting $\underline a=(a_0,
\ldots, a_{\alpha})\in \mathbb K^{\alpha+1}$ and $g_{\underline a}=\sum_{i=0}^{\alpha}a_ig^i$,

\begin{equation}\label{linTz}
r=\sum_{i=0}^{\alpha}a_i\psi_{g^i}(p)=\psi_{g_{\underline a}}(p).
\end{equation}

Let us recall a fundamental result proved by Gordan and Noether
(see \cite{GN} and \cite[2.7]{Lo}).

\begin{thm}\label{partialdiff}{\rm ({\bf Gordan--Noether Identity})}
Let notation be as above and let $F\in \mathbb K[x_0,\ldots ,x_N]$. Then:
\begin{equation}\label{2.7}
\sum_{i=0}^N\frac{\partial F}{\partial x_i}h_i=0\,
\Leftrightarrow\, F(\mathbf x)=F(\mathbf x+\lambda\psi_g(\mathbf x))\,\,\, \forall\lambda\in \mathbb K.
\end{equation}
\end{thm}

A  result used many times in the sequel is the following immediate consequence of the Gordan-Noether
Identity.

\begin{cor}\label{thm:polar&perazzo}
  Let $X = V(f) \subset \P^N$ be a  hypersurface with vanishing hessian and let notation be as above. Then
\begin{enumerate}
\item[i)] for every $p\in\p^N\setminus \Sing X$ such that $\Phi_X(p)\in Z_{\reg}$ we have
$<p,(T_{\Phi_X(p)}Z)^*>\subseteq \Phi_X^{-1}(\Phi_X(p))$;
\medskip
\item[ii)] for  $p\in\p^N$ general,  the irreducible component of $\overline{\Phi_X^{-1}(\Phi_X(p))}$
passing through $p$ is $<p,(T_{\Phi_X(p)}Z)^*>.$
In particular for $p\in\p^N$ general $\overline{\Phi_X^{-1}(\Phi_X(p))}$ is a union of linear spaces of dimension equal
to $\codim(Z)$
passing through $(T_{\Phi_X(p)}Z)^*$.
\medskip
\item[iii)] 
\begin{equation}\label{inclZ*}
Z^*\subseteq \Sing X.
\end{equation}
\medskip
\end{enumerate}
\end{cor}
\begin{proof}
Note that 
\begin{equation}\label{partial2f}
\sum_{i=0}^N\frac{\partial ^2f}{\partial x_i\partial
x_j}h_i=0\mbox{ for every\;\; } j=0,\ldots,N. 
\end{equation}

 This relation is obtained by  differentiating the equation
$g(\frac{\partial f}{\partial x_0},\ldots,\frac{\partial f}{\partial x_N})=0$ with respect to $x_j$ and applying the chain
rule. As a consequence, we get the following relation by Theorem
\ref{partialdiff}:
\begin{equation}\label{2.5}
   \frac{\partial f}{\partial x_i}(\mathbf x)=\frac{\partial f}{\partial x_i}(\mathbf x+\lambda\psi_g(\mathbf x))
   \mbox{ for every\;\; } j=0,\ldots,N,
\end{equation}
for every $\lambda\in\mathbb K$ and for every $g\in \mathbb K[y_0,\ldots, y_N]$ such that $g(\frac{\partial f}{\partial x_0},\ldots,\frac{\partial f}{\partial x_N})=0.$

Let  $p\in\p^N\setminus \Sing X$ such that $\Phi_X(p)\in Z_{\reg}$.  Let $r\in (T_{\Phi_X(p)}Z)^*$ be a general
point. By \eqref{linTz} we can suppose $r=\psi_g(p)$ so that \eqref{2.5} yields  that the line $<p,r>$ is contracted to the point $\Phi_X(p)$
and that $r=<p,r>\cap \Sing X$.
The generality of $r$ implies that the linear space $\p^{\codim(Z)}=<p,(T_{\Phi_X(p)}Z)^*>$
is contained in   
$\Phi_X^{-1}(\Phi_X(p))$, proving i) and also that $(T_{\Phi_X(p)}Z)^*\subset \Sing X$. Moreover since, by definition,  $Z^*$  is ruled by  the linear spaces  $(T_{\Phi_X(p)}Z)^*$ part iii) immediately follows.

If $p\in\p^N$ is general, then, by generic smoothness, the irreducible component of $\Phi_X^{-1}(\Phi_X(p))$ passing
through $p$  has dimension $\codim(Z)$ and it is smooth at $p$ so that it coincides with $<p,(T_{\Phi_X(p)}Z)^*>$, proving ii). 
\end{proof}

\section{The Perazzo map of  hypersurfaces with vanishing hessian}\label{Perazzomap}

Let us recall the so called {\it Reciprocity Law of Polarity} to be used later on in  the analysis of the geometry of $Z^*$.
We define first the notion of {\it degree  $s$ Polar hypersurface of $X=V(f)\subset\p^N$}.

\begin{defin}{\rm For every $s=1,\ldots, d-1$ and for every $p\in\p^N$ the {\it degree $s$ Polar of $X$ with respect to $p$} is the hypersurface
$$H^{s}_p(f):=V(\sum_{i_0+\ldots+i_N=s}\frac{\partial^sf}{\partial x_0^{i_0}\ldots x_N^{i_N}}(p)x_0^{i_0}\ldots x_N^{i_N})\subset\p^N.$$
By definition $\deg(H^s_p(f))=s$ if the polynomial on the right in the above expression is not identically zero. 
 Otherwise we naturally put  $H^s_p(f)=\p^N$.
For $s=1$ the hyperplane $H^1_p(f)$ is the polar hyperplane of $X$ with respect to $p$, which will be 
indicated simply by $H_p$. For $s=2$ the hypersurface $H^2_p(f)$ is a quadric hypersurface whose associated symmetric matrix is $\Hess_X(p)$ and we shall put, by abusing notation, $Q_p=H^2_p(f)$ if
the reference to $f$ is well understood. 
}
\end{defin}

We recall a classical result used repeatedly in the sequel. For a proof and other applications one can consult the first chapter of \cite{Do1}.
\medskip

\begin{prop}\label{reciprocity}{\rm({\bf Reciprocity Law of Polarity})} Let $X=V(f)\subset\p^N$ be a degree $d$ hypersurface. Then for every $s=1,\ldots, d-1$ and for every distinct points $p,q\in\p^N$  we have
$$p\in H^s_q(f)\;\;\iff q\in H^{d-s}_p(f).$$

In particular we have 
\begin{equation}\label{multrec}
\{p\in X\;:\;\mult_p(X)\geq s\}= \bigcap_{q\in\p^N}H^{d-s+1}_q.
\end{equation}
\end{prop}
\medskip

Perazzo introduced in \cite{Pe} the notion of (Perazzo) rank of a cubic hypersurface with vanishing hessian, which we now extend to the general case. Although he does not explicitly
define the rational  map described below,  its use  was implicit in his analysis, see {\it loc. cit}.

\begin{defin} \rm
 Let $X = V(f) \subset \P^N$ be a reduced hypersurface with vanishing hessian, let $\Phi_X:\P^N \dashrightarrow \P^N$ be  its polar map and  let 
$Z =\overline{ \Phi_X(\P^N)}\subsetneq (\p^N)^*$ be its polar image. {\it The Perazzo map of $X$} is the rational map:
$$\begin{array}{cccc} \mathcal{P}_X: & \P^N & \dashrightarrow & \mathbb{G}(\codim(Z)-1,N) \\
\ & p & \mapsto & (T_{\Phi_X(p)}Z)^* \end{array}$$ defined as above in the open set $\mathcal{U} = \Phi_X^{-1}(Z_{\reg})$,
where $Z_{\reg}$ is the locus of smooth  points of $Z$. 

With this notation we have that for $p\in \p^N$ general point $\Sing Q_p=(T_{\Phi_X(p)}Z)^*$. Moreover by definition
$\mathcal P_X$ is the composition of $\Phi_X$ with the dual Gauss map of $Z$ introduced in Definition \ref{dualGaussmap}.

The image of the Perazzo map will be denoted by $W_X=\overline{\mathcal{P}_X(\P^N)} \subset \mathbb{G}(\codim(Z)-1,
N)$, or simply by $W$,  and  $\mu = \dim W$ is called the {\it Perazzo rank of $X$}.
\end{defin}

We shall always identify $(\p^N)^{**}$ with the original $\p^N$ so that, if $\codim(Z)=1$, we also use
the identification $\mathbb G(0,N)=\p^N$. Let us remark that by definition and with the previous identifications we have 
\begin{equation}\label{Zdualdef}
Z^*=\displaystyle\overline{\bigcup_{z\in Z_{\reg}}(T_zZ)^*}\subset \p^N.
\end{equation}
\medskip

\begin{rmk}\label{rm:mu0}{\rm If $\mu=0$, then $Z\subset(\p^N)^*$ is a linear space 
and
$X$ is a cone such that $\Vert(X)=Z^*=\p^{\codim(Z)-1}\subseteq X$. 

Therefore if $X=V(f)\subset\p^N$ is a hypersurface with vanishing hessian,  not a cone, then $\mu\geq 1$.

If $\codim(Z)=1$ and if $Z=V(g)$, then $\mathcal P_X=\psi_g$, where $\psi_g$ is the rational map introduced
in Definition \ref{defpsig}, and $W=Z^*$.
}
\end{rmk}

The following result  will be useful to determine  the  structure of particular classes of  cubic hypersurfaces
with vanishing hessian defined in the sequel. First let us remark that  for a cubic hypersurface  and for $r,s\in\p^N$ we have: 
$r\in\Sing Q_s$ if and only if $s\in\Sing Q_r$.

\begin{thm} \label{Perazzolinear}
 Let $X = V(f) \subset \P^N$ be a  cubic hypersurface with vanishing hessian. 
 Let $w = [(T_{\Phi_X(p)}Z)^*] \in W_X\subset \mathbb{G}(\codim(Z)-1,N)$ be a general point
 and let $r \in (T_{\Phi_X(p)}Z)^*$ be a general point with $p\in\p^N$ general.
Then:
 \begin{equation}\label{linearfiber}
 \overline{\mathcal P_X^{-1}(w)}=\bigcap_{r\in (T_{\Phi_X(p)}Z)^*=\Sing Q_p} \Sing Q_r=\p^{N-\mu}_w.
 \end{equation} 
 \end{thm}
\begin{proof} By definition $\overline{\mathcal P_X^{-1}(w)}$ is the closure of  the set of all (general) $p'\in\p^N$ such that ${\rm Sing}(Q_{p'})={\rm Sing}(Q_{p})$. This happens if and only if
$r\in {\rm Sing}(Q_{p})$ implies $r\in {\rm Sing}(Q_{p'})$ (or vice versa by symmetry and by the generality of $p,p'\in\p^N$), which in turn happens if and only if $p'\in {\rm Sing}(Q_{r})$ for every $r\in \Sing Q_p$, yielding the first equality in \eqref{linearfiber} and concluding the proof.
\end{proof}
\medskip

\begin{rmk}\label{rm:linearfiber}
{\rm

\begin{enumerate}
\item[i)] In \cite[\S 5, pg. 339]{Pe} Perazzo states that $\overline{\mathcal P_X^{-1}(w)}$
is a linear space of dimension $N-\mu$ but in our opinion his arguments contain some gaps.

Indeed, if  $\codim(Z)=1$, $(T_{\Phi_X(p)}Z)^*=r$ is a point and \eqref{linearfiber} gives
$\overline{\mathcal P_X^{-1}(w)}=\Sing Q_r$.
Perazzo wrongly claims that $\Sing Q_r=\overline{\mathcal P_X^{-1}(w)}$  for $r\in (T_{\Phi_X(p)}Z)^*$ general also when  $\dim((T_{\Phi_X(p)}Z)^*)=\codim(Z)-1>0$. To prove this wrong claim  Perazzo erroneously assumes that the tangent space $T_rZ^*$ is constant for $r\in (T_{\Phi_X(p)}Z)^*$ general  when $\dim(T_{\Phi_X(p)}Z)^*)>0$. This  claim would be  equivalent to $\dim(Z^*)=\codim(Z)-1+\mu$ as we shall see in  Proposition \ref{estimateF} below, a fact used by Perazzo repeatedly in his paper.

Perazzo probably  confuses the contact locus of a general
tangent space to $Z^*$, that is the closure of a general fiber of the Gauss map of $Z^*$, with the contact locus
of  a general  hyperplane tangent to $Z^*$, which is a linear space  of the form $(T_{\Phi_X(p)}Z)^*$. For $\codim(Z)>1$ these contact loci can differ and  the linear spaces $\Sing Q_r$ can depend on  $r\in (T_{\Phi_X(p)}Z)^*$. 
\medskip

\item[ii)] Formula \eqref{linearfiber} is a particular case of general results on linear systems
of quadrics  {\it with tangential defect}, see \cite{Degoli}, \cite{Landsberg}
and particularly  \cite[Corollary 1]{Alzati} where a proof of \eqref{linearfiber} is deduced 
from these general facts.
As shown by the examples in the last pages of \cite{Alzati} the fibers
of the Perazzo map are not necessarily linear for hypersurfaces of degree greater
than three. The general structure of the fibers of the Perazzo map for arbitrary hypersurfaces
is, to the best of our knowledge, unknown.

\end{enumerate}
} 
\end{rmk}
\medskip

In order to better clarify the wrong claims made by Perazzo and outlined in part i) of the previous Remark we collect
some facts which will be also used later.
\medskip

\begin{prop}\label{estimateF} Let notation be as above and let $F=\P^{\dim(Z)-\mu}\subset Z$ be the closure of the fiber of the Gauss map of $Z$ passing through the general point $\Phi_X(p)\in Z$. Then:
\begin{enumerate}
\item[a)] $\dim(Z^*)\leq\codim(Z)-1+\mu$;
\medskip
\item[b)] if $r\in (T_{\Phi_X(p)}Z)^*\subset Z^*$ is a general point, then
\begin{equation}\label{FinclTrZ}
F\subseteq (T_rZ^*)^*
\end{equation} 
with equality if and only if $\dim(Z^*)=\codim(Z)-1+\mu$. 
\medskip

\item[c)] If   $\codim(Z)=1$, then  $F^*=T_rZ^*$.
\medskip
\item[d)] If $w=[(T_{\Phi_X(p)}Z)^*]$, then
\begin{equation}\label{intHq}
 \bigcap_{q \in \p^{N-\mu}_w } H_q = F^*\supseteq T_rZ^*. 
\end{equation}
\end{enumerate}
\end{prop}
\begin{proof} Let us remark that \eqref{Zdualdef} yields $\dim(Z^*)\leq\codim(Z)-1+\mu$ with equality 
if and only if through a general point $r\in Z^*$ there passes a finite number of  linear spaces of the form  $(T_{\Phi_X(p)}Z)^*$,
proving a).

Let  $r\in (T_{\Phi_X(p)}Z)^*\subset Z^*$ be a general point. 
By Reflexivity  $(T_rZ^*)^*$ is the contact locus on $Z$ of the hyperplane $r^*\in Z$, tangent to $Z$ at the point
$\Phi_X(p)$ (to deduce this one remarks that  $r\in (T_{\Phi_X(p)}Z)^*$). Thus, since  $F$ is the contact locus on $X$ of $T_{\Phi_X(p)}Z$,
we get 
$$
F\subseteq (T_rZ^*)^*,
$$
with equality if and only if $\dim(Z^*)=\codim(Z)-1+\mu$. 
Furthermore if   $\codim(Z)=1$, then  $F^*=T_rZ^*$ and the proofs of b)  and c) are complete. 

Moreover, if $\dim(Z^*)=\codim(Z)-1+\mu$, we claim that there exists a unique linear space of the form $(T_{\Phi_X(p)}Z)^*$ passing through a general point
of $Z^*$. Indeed since 
$F^*=T_rZ^*$, $T_rZ^*$  is independent of $r\in (T_{\Phi_X(p)}Z)^*$. Therefore $(T_{\Phi_X(p)}Z)^*$
is contained in the contact locus on $Z^*$ of the general tangent space $T_rZ^*$ from which it follows that $(T_{\Phi_X(p)}Z)^*$
is the (closure of the) fiber of the Gauss map of $Z^*$ passing through $r$, proving the claim.

If $H$ is a general hyperplane such that $H \supseteq F^*$, 
then  $[H] \in (F^*)^*=F\subset  Z$ so that  $[H]=\Phi_X(q)$ with $q\in\p^{N-\mu}_w$ and $H=H_q$. Thus 
$$
\dy \bigcap_{q \in \p^{N-\mu}_w } H_q = F^*\supseteq T_rZ^*,
$$
proving $d)$.
\end{proof}
\medskip

\begin{thm} \label{cones}
 Let $X = V(f) \subset \P^N$ be an irreducible  cubic hypersurface with vanishing hessian, not a cone, with $\codim(Z)=1$. 
 Let $r = [(T_{\Phi_X(p)}Z)^*] \in W=Z^*\subset \mathbb{P}^{N}$
 with $p\in\p^N$ general 
 and let $X_r = \overline{\mathcal P_X^{-1}(r)} \cap X$, which is equidimensional of dimension
$N-\dim(Z^*)-1$.
Then:
\begin{equation}\label{Perazzoestimate}
\dim(Z^*)\leq\frac{N-1}{2}.
\end{equation}
\end{thm}
\begin{proof} 
Let us consider the cubic hypersurface $X_r=\P^{N- \dim(Z^*)}_r \cap X$, which could also 
have some non reduced component.

Let  $\tilde{H}_p$ be 
the polar hyperplane of a point $p \in \P_r^{N- \dim(Z^*)}$ with respect to $X_r$. One immediately sees  that 
$\tilde{H}_p = H_p \cap \P^{N- \dim(Z^*)}_r $. Since $\codim(Z)=1$, letting  $F$ be the closure of the fiber of the Gauss map of $Z$ passing through $\Phi_X(p)$, we have  $F^*=T_rZ^*$  by part b) of Proposition \ref{estimateF} so that by \eqref{intHq}
$$\bigcap_{p \in \P_r^{N- \dim(Z^*)}} H_p=T_rZ^*.$$
Moreover
$$T_rZ^*\subseteq T_r\Sing X=\Sing Q_r=\p^{N-\dim(Z^*)}_r,$$
where the first inclusion follows from $Z^*\subseteq\Sing X$, the first equality from
definition of the scheme $\Sing X$ and the last equality from  \eqref{linearfiber} applied to $W=Z^*$.

Therefore
\begin{equation}\label{tripleXp}
\dy \bigcap_{p \in \P_r^{N- \dim(Z^*)}} \tilde{H}_p = (\dy \bigcap_{p \in \P_r^{N- \dim(Z^*)}} H_p ) \cap \p^{N-\dim(Z^*)}_w = T_rZ^* \cap \p^{N-\dim(Z^*)}_r =T_rZ^*.
\end{equation}

By Proposition \ref{reciprocity} the linear space $T_rZ^*$ is exactly the locus of points of multiplicity three of the cubic hypersurface $X_r\subset\p^{N-\dim(Z^*)}_r$. Thus  the cubic hypersurface $X_r\subset \p^{N-\dim(Z^*)}$ is a cone whose vertex is $T_rZ^*$.  In particular $\dim(Z^*)=\dim(T_rZ^*)\leq N-\dim(Z^*)-1$, proving
\eqref{Perazzoestimate}. 
\end{proof}
\medskip

The linearity of the fibers has strong consequences on the geometry of a cubic hypersurface with vanishing hessian. Let us recall
some easy and well known facts on {\it congruences of order one of linear spaces}, that is irreducible families $\Theta\subset\mathbb G(\beta, N)$
of linear spaces of dimension $\beta>0$ such that through a general point of $\P^N$ there passes a unique member of the family.
Let us remark that the previous condition forces $\dim(\Theta)=N-\beta$ and that the tautological map $p:\mathcal U\to \p^N$
from the universal family $\pi:\mathcal U\to \Theta$ is birational onto $\p^N$.

Let notation be as above and let 
$$V=\{q\in\p^N\,;\, \#(p^{-1}(q))\geq 2\}=\{q\in\p^N\,;\, \dim(p^{-1}(q))>0\}\subset\p^N$$
be the so called {\it jump (or branch) locus of $\Theta$}.
\medskip

The easiest examples of congruences of linear spaces of dimension $\beta$ is given by the family of linear spaces
of dimension $\beta+1$ passing through a fixed linear space $L=\p^beta\subset\p^N$.

The previous examples motivate the following definition.

\begin{defin}\rm
An irreducible cubic hypersurface $X\subset\p^N$ with vanishing hessian, not a cone, will be  called a {\it Special Perazzo Cubic Hypersurface} if the general fibers of its Perazzo map form a congruence of linear spaces passing through
a fixed $\p^{N-\mu-1}$. 
\end{defin}
\medskip

Special Perazzo Cubic Hypersufaces will be treated  in Section \ref{examplesmall} following very closely the original treatment  of Perazzo
and providing canonical forms and explicit geometrical descriptions in any dimension.
Let us remark that for  Special Perazzo Cubic Hypersurfaces the linear span of two general fibers is a $\p^{N-\mu+1}$
which is strictly contained in $\p^N$
if $\mu>1$. For $\mu=1$ the general fibers determine a pencil of hyperplanes and this case will  be treated 
in arbitrary dimension in Theorem \ref{mu = 1}.

To ensure that a cubic hypersuface is a Special Perazzo Cubic Hypersurface it is sufficient to control the linear span (or equivalently
the intersection) of two general fibers. Indeed we have the following easy result, whose proof is left to the reader.
\medskip

\begin{lema}\label{sistema especial de Perazzo}
  Let $X = V(f) \subset \P^N$ be an irreducible cubic hypersurface with vanishing hessian, not a cone.
Let $w_1,w_2 \in W$ be general points, let $\overline{\mathcal P_X^{-1}(w_i)}=\p^{N-\mu}_{w_i}$ be the corresponding fibers of the Perazzo map. Suppose that $\p^{N-\mu}_{w_1} \cap \p^{N-\mu}_{w_2} = \P^{N-\mu-1}$. 

Then $X\subset\p^N$ is a Special Perazzo Cubic Hypersurface.
\end{lema}

The linear span of two general fibers $\p^{N-\mu}_{w_1}$ and $\p^{N-\mu}_{w_2}$ of the Perazzo map is strongly related
to the dimension of the secant variety $SZ^*\subseteq\p^N$ of $Z^*$ at least for $\codim(Z)=1$. For cubic hypersurfaces
with vanishing hessian we have an obvious inclusion
$SZ^*\subseteq X$ since $Z^*\subseteq \Sing X$. 
\medskip

The following result, surely well known (see for example \cite{C.Segre3}),  has many applications in this context. The proof
can be left to the reader.
\medskip
\begin{lema}\label{pencils} Let $\{Q_\lambda\}_{\lambda\in\p^m}$ be a linear system of  quadric hypersurfaces in $\p^N$. Then
\begin{equation}\label{spanSing}
<\overline{\bigcup_{\lambda\in\p^m\mbox{ general}}\Sing Q_\lambda}>\subset Q_t
\end{equation}
for $t\in\p^m$ general.
\end{lema}

The next Proposition was inspired by \cite[footnote  pg. 348,349]{Pe} and it relies essentially on the previous geometrical
fact.

\begin{prop} \label{prop:lema1,2}
  Let $X=V(f) \subset \P^N$ be an irreducible cubic hypersurface with vanishing hessian, not a cone.
 Then the following conditions are equivalent:
\begin{enumerate}
  \item[i)] $SZ^* \subseteq \Sing X$;
  \medskip
  \item[ii)] $Z^* \subseteq \displaystyle \bigcap_{w \in W\mbox{{\rm general}}}\p^{N-\mu}_w$;
  \medskip
    \item[iii)] $<Z^*> \subseteq \displaystyle \bigcap_{w \in W\mbox{{\rm general}}}\p^{N-\mu}_w$;
  \medskip
  \item[iv)] $<Z^*> \subseteq \Sing X$.
\end{enumerate}
\end{prop}

\begin{proof}
 Let us suppose that $SZ^* \subseteq \Sing X$. If $r_1, r_2 \in Z^*$ are general points, then $< r_1, r_2 > \subseteq
  \Sing X$ implies   $r_2 \in \operatorname{Sing} Q_{\epsilon}$  for every $\epsilon \in
  <r_1,r_2>$. Taking $r_1$ general, $r_1 \in (T_zZ)^*$ where $w=[(T_zZ)^*] \in
  W$ is general, then
  $r_2 \in \operatorname{Sing} Q_{r_1}$, which by the generality of $r_1\in(T_zZ)^*$
  yields
  $$r_2\in\bigcap_{r_1\in (T_zZ)^*}\Sing Q_{r_1}= \p^{N-\mu}_w.$$
By the generality of $r_2 \in Z^*$ and of $w \in W$ we deduce:
$$Z^* \subseteq \dy \bigcap_{w \in W}\p^{N-\mu}_w.$$
Of course ii) is equivalent to iii).

Let us suppose now 
$ <Z^*>  \subseteq \dy \bigcap_{w \in W{\rm general}}\p^{N-\mu}_w$ and  let $s \in <Z^*>$ be a general point. By hypothesis and by 
\eqref{linearfiber} we have that  $s \in
\operatorname{Sing} Q_r$ with  $r \in Z^*$  general. 
We claim that  $$S^k Z^* \subseteq \Sing X \Rightarrow S^{k+1} Z^* \subseteq \Sing X$$
from which iv) will follow by induction.

Since  $S^{k+1} Z^* = S(S^k Z^*,Z^*)$, we take   $q\in S^k Z^* \subset \Sing X$ and $r \in Z^*\subset \Sing X$ general points. Then $q\in <Z^*>$
implies $q\in \Sing Q_r$  so that 
$\frac{\partial f}{\partial x_i} (\lambda q + \mu r) 
= \frac{\partial f}{\partial x_i} (q) \lambda^2 + \frac{\partial f}{\partial x_i} (r) \mu^2\equiv 0$ since  $q,r \in \Sing X$. 
Therefore $\lambda q + \mu r \in \Sing X$, which by the generality of $q$ and $r$ yields $S^{k+1} Z^* \subset \Sing X$, as claimed.

 The last implication iv) $\Rightarrow$ i) is obvious.
\end{proof}
\medskip

\begin{ex}\label{Segreproj} \rm Consider the Segre variety  $Y = \operatorname{Seg}(1,2)=\p^1\times\p^2 \subset \P^5 = \P(\mathbb{M}_{2 \times 3}(\mathbb K))$. 
With the last identification we have
$$\operatorname{Seg}(1,2) = \{ [A]\ |\ \rk(A) = 1 \}$$ and $SY=\p^5$.

The projection of $Y \subset \P^5$ from a point $p\in\p^5\setminus Y$ is a cubic hypersurface 
$X = V(f) \subset \P^4$. It is easy to check that $(\Sing X)_{\red} = \P^2$ because the secant lines to
$Y$ passing through $p$ describe a two dimensional quadric surface $\Sigma_p\subset Y$. 

We claim that
$X$ has vanishing hessian. Indeed, by the Reciprocity Law of Polarity, the polar quadric $Q_p$ of
a general point $p\in\p^4$ is a cone being a 3-dimensional quadric containing the plane $(\Sing X)_{\red}$,
see also  Proposition \ref{prop:classe1} for generalizations of this argument. Furthermore  $X$ is not a cone since  its dual is clearly non degenerate being a general hyperplane 
section of $Y^* \simeq Y$. The example recalled in the Introduction is projectively equivalent to $X$.
Indeed, modulo projective equivalence, $V(x_0x_3^2 +  x_1x_3x_4 +  x_2x_4^2)\subset \p^4$ is  the only  cubic hypersurface in $\p^4$ with vanishing hessian and not a cone,
see Theorem \ref{P4}. 

Since $Z^*\subset (\Sing X)_{\red}=\p^2=\Pi$, $Z\subset\p^4$ is a cone over the dual curve of $Z^*$ as a plane curve and
whose vertex is the line $L=\Pi^*$. Thus $\codim(Z)=1$ and $\mu=1$. One easily deduces that 
$X\subset\p^4$ is
a Special Perazzo Cubic Hypersurface whose associated congruence is given by the family of $\p^3$'s passing through $(\Sing X)_{\red}$, see also  Proposition \ref{prop:lema1,2}.
\end{ex}
\medskip

\begin{prop}\label{geramPN}
  Let $X = V(f) \subset \P^N$ be an irreducible cubic hypersurface with vanishing hessian, not a cone, with $\codim(Z)=1$.
Let $w_1,w_2 \in W=Z^*$ be general points and let $\overline{\mathcal P_X^{-1}(w_i)}=\p^{N-\mu}_{w_i}$ be the corresponding fibers of the Perazzo map. 
If $<\p^{N-\mu}_{w_1},\p^{N-\mu}_{w_2}>=\p^N$, then $<Z^*>\subseteq \Sing X$.
\end{prop}
\begin{proof} Let $p\in \p^N$ be a general point. Then there exist general $p_i\in
\p^{N-\mu}_{w_i}$ such that $p\in<p_1,p_2>.$ Since $w_i=\mathcal P_X(p_i)=\Sing Q_{p_i}$,
 Lemma \ref{pencils} implies $<Z^*>\subseteq Q_p$ for $p\in\p^N$ general. To conclude
it suffices to recall that 
$$(\Sing X)_{\red}=\bigcap_{p\in\p^N\mbox{{\rm general}}} Q_p,$$
 by Proposition \ref{reciprocity}.
\end{proof}
\medskip

When $\codim(Z)=1$ there are interesting relations between the base locus of the Perazzo map and $W=Z^*$. Indeed
if $Z=V(g)$, letting notation be as in Definition \ref{defpsig}, we have 
$$\mathcal P_X=\psi_g=(h_0:\ldots:h_N):\p^N\dasharrow\p^N$$
and $Z^*=W=\overline{\psi_g(\p^n)}$.
\medskip

\begin{prop}\label{inclZdualL} Let $X=V(f)\subset\p^N$ be an irreducible cubic hypersurface with vanishing hessian,
not a cone, with $\codim(Z)=1$ and let $\Bs(\mathcal P_X)=\Bs(\psi_g)=V(h_0,\ldots,h_N)$ be the base locus scheme of $\mathcal P_X$. Then:
\begin{enumerate}
\item[a)] $Z^*\subseteq \Bs(\mathcal P_X)$;
\medskip
\item[b)] if $X$ is a Special Perazzo Cubic Hypersurface, then
\begin{equation}\label{inclZ*L}
Z^* \subseteq \displaystyle \bigcap_{w \in W\mbox{{\rm general}}}\p^{N-\mu}_w=\p^{N-\mu-1}\subseteq\Bs(\mathcal P_X).
\end{equation} 
\end{enumerate}
\end{prop}
\begin{proof} Let notation be as in Subsection \ref{GNIDS} with $Z=V(g)\subset\p^{N*}$. Since $g_i=\rho\cdot h_i$ for every $i=0,\ldots, N$
and since $g_i(\mathbf x+\lambda\psi_g(\mathbf x))=g_i(\mathbf x)$ for every $i=0,\ldots, N$ and for every $\lambda\in\mathbb K$, we deduce  $h_i(\mathbf x+\lambda\psi_g(\mathbf x))=h_i(\mathbf x)$ for every $i=0,\ldots, N$ and for every $\lambda\in\mathbb K$. Thus
\begin{equation}\label{GNPX}
\psi_g(\mathbf x+\lambda\psi_g(\mathbf x))=\psi_g(\mathbf x)
\end{equation}
for every $\lambda\in\mathbb K$. In particular  we have $h_i(\psi_g(p))=0$ for every $i=0,\ldots, N$, proving a).

To prove b) let us remark that $\p^{N-\mu}_w\cap \Bs(\mathcal P_X)$ is a hypersurface of degree $e=\deg(h_i)\geq 1$
and that, by hypothesis, the intersection of two general $\p^{N-\mu}_{w_i}$'s is a $\p^{N-\mu-1}=L$ contained in $\Bs(\mathcal P_X)$ and in a general $\p^{N-\mu}_w$. Thus $L$ is an irreducible component of $\p^{N-\mu}_w\cap \Bs(\mathcal P_X)$. Let $p\in \p^{N-\mu}_w$ be a general point. By \eqref{GNPX} the line $<p,\psi_g(p)>$, which is contained in $\p^{N-\mu}$, cuts $\Bs(\mathcal P_X)$ only in the point $\psi_g(p)$, yielding  $\psi_g(p)\in L$ because
$<p,\psi_g(p)>\cap L\neq\emptyset$. The generality of $w$ implies $\psi_g(p)\in L$ for $p\in\p^N$ general so that $Z^*\subseteq L$.
\end{proof}

\section{Classes of cubic hypersurfaces with vanishing hessian according to Perazzo and canonical forms
of Special Perazzo Cubic Hypersurfaces}\label{examplesmall}

The examples of hypersurfaces with vanishing hessian not cones considered in \cite{GN} and in Perazzo \cite{Pe} 
are singular along linear spaces, see also \cite{CRS}.
We now prove that in some cases the existence of a linear space of sufficiently high dimension in the singular locus assures that a cubic hypersurface
has vanishing hessian.
\medskip

\begin{prop}\label{prop:classe1}
 Let $X\subset \P^N$, $N \geq 4$, be an irreducible cubic hypersurface. 
Suppose that there exists a linear space $S=V(x_{\tau+1},\ldots,x_N)=\P^{\tau}$ contained in 
the singular locus of $X$.  Then, up to projective transformations:
\begin{equation}\label{cancub}
f = \dy \sum_{i=0}^{\tau}x_iC^i(x_{\tau + 1}, ..., x_N) + D(x_{\tau + 1}, ..., x_N),
\end{equation}
with $C^i(x_{\tau + 1}, ..., x_N)$ a quadratic form for every $i=0,\ldots,\tau$ and with
$D(x_{\tau + 1}, ..., x_N)$ a cubic form.

\begin{enumerate}
 \item[i)] If $ \tau > \frac{N-1}{2}$, then $\hess_X  = 0$; 
\medskip
\item[ii)] If $\tau >\frac{(N-\tau)(N-\tau+1)}{2}-1$, then $X$ is a cone.
\end{enumerate}
In particular, for $N > 4$  every non normal cubic hypersurface is a cone.
\end{prop}

\begin{proof} The proof of \eqref{cancub} is left to the reader.
To prove i) it is enough to remark that $S\subseteq\Sing X\subset Q_p$ with $p\in\p^N$ general by the Reciprocity Law of Polarity.
If $\tau>\frac{N-1}{2}$, then $Q_p$ is a cone and the generality of $p$ implies $\hess_X= 0$.

To prove ii)  we observe that $f_i=C_i$ for every $i=0,\ldots, \tau$ and that the number of independent quadrics in $N- \tau$ variables is $\left (\begin{array}{c} N - \tau + 1 \\ 2 \end{array} \right)$.

The cubic hypersurface is non normal if and only if  to $\dim (\Sing X) = N-2 $, which on the other hand implies
$(\Sing X)_{\red}= \P^{N-2}$. Therefore we can assume $N>4$ and $\tau=N-2$ so that   $X$ is a cone by part ii).
\end{proof}

\begin{rmk}\label{singp2}
  { \rm For cubic hypersurfaces in $\P^5$ the existence of a singular plane does not imply the vanishing of the hessian. Indeed, if 
  $$X=V(x_0x_3^2 + x_1x_4^2 + x_2x_5^2)\subset\p^5$$
then $(\Sing X)_{\red}=V(x_3,x_4,x_5)=\p^2$ but $\hess_X\neq 0$.
 Thus the  bound in Proposition \ref{prop:classe1} is sharp for any dimension
by easy manipulations on the previous example.

Via direct computations one can verify that for a general $p\in \p^5$ the non-singular quadric $Q_p$ contains the linear space
$S=V(x_3,x_4,x_5)$ and that $\cap_{s \in S} T_sQ_p =S$.}
\end{rmk}

We shall now introduce some terminology in order to obtain a projective characterization
of  Special Perazzo Cubic Hypersurfaces. 
 
This approach follows strictly  the original work of Perazzo, who used the next results to provide
simplified canonical forms for Special Perazzo Cubic Hypersurfaces, see \cite[\S 14-18 p. 344-350]{Pe}.

The first definition is a generalization of the classical notion of {\it generatrix} of a cone with vertex a 
point, that is a line contained in the  cone and passing through the vertex.

\begin{defin} \rm
 Let $Q \subset \P^N$ be a a quadric of rank $\operatorname{rk}(Q) = \beta\geq 3$ and let
$\operatorname{Vert}(Q) = \operatorname{Sing} Q = \P^{N- \beta}$ be its vertex. A linear space $M = \P^{\tau}$, $\tau\geq 1$, is a {\it generator of 
$Q$} if $$\operatorname{Vert}(Q) \subsetneq M \subset Q.$$
\end{defin}

The proof of the next result is left to the reader.
\medskip

\begin{lema}\label{tangencia}
  Let $Q \subset \P^N$ be a quadric of rank $\operatorname{rk}(Q) =\beta$ and let $M =\P^{\tau}$ 
be a generator of $Q$. Then
  $$\dy \bigcap_{m \in M\setminus\Vert(Q)} T_mQ = \p(\Phi_Q(M))^*=\P^{{2N-\beta-\tau}}.$$
 \end{lema}
\medskip

\begin{defin}\rm
Let $Q \subset \P^N$ be a quadric, let
$M$ be a generator of $Q$ and let $L$ be a linear space containing $M$.
The quadric is {\it tangent to the linear space $L$ along the generator $M$} if 
$L \subseteq T_mQ$ for all $m \in M$.
\end{defin}

Obviously a necessary condition for $L$ being tangent to $Q$ along $M=\p^\tau$ is 
$$L\subseteq\bigcap_{m \in M\setminus\Vert(Q)} T_mQ=\p^{2N-\rk(Q)-\tau},$$
yielding  $\dim(M)=\tau\leq\dim(L)\leq 2N-\rk(Q)-\tau$.

In the example analyzed in Remark \ref{singp2} the general polar quadric $Q_p\subset\p^5$, which is nonsingular,  is tangent to $M=V(x_3,x_4,x_5)=(\Sing X)_{\red}$ along $M$. 
\medskip

The next result is the core of the work of Perazzo on canonical forms. We shall always assume $\codim(Z)=1$
or equivalently $\rk(Q_p)=N$ for $p\in\p^N$ general which corresponds to the case $p=0$  considered
by Perazzo. We shall obtain  a simplified canonical 
form refining Proposition \ref{prop:classe1}, corresponding to the case in which the polar quadric of a general point of $\P^N$ 
is tangent to a fixed linear space $\P^{N - \tau}$ along a common generator of dimension $\tau$, see \cite[\S 1.17 p.349]{Pe}.

We shall consider $\operatorname{Sing} Q_p \subset M \subset Q_p$ with $M=\p^\tau$ a common  generator having  equations: $x_{\tau + 1} = ... = x_N = 0$. 
Furthermore we shall suppose 
$$L=\dy \bigcap_{m \in M\setminus\Vert(Q_p)} T_mQ_p = \P^{N-\tau},$$
that is $L$ will always be  the maximal linear subspace containing $M$ along which the quadric $Q_p$ can be tangent along the generator $M$.
Moreover we shall assume that 
$L$ has equations: $x_{N-\tau + 1} = ... = x_N = 0$. Under these hypothesis $\dim(L)=N-\tau\geq\tau=\dim(M)$ with equality
if and only if $L=M$.

\begin{thm}\label{forma_canonica_sist_linear}
 Let $X= V(f) \subset \P^N$ be a cubic hypersurface
 having vanishing hessian, not a cone, and with $\codim(Z)=1$.
 Suppose  that 
the polar quadric of a general   point of $\P^N$ is  tangent to a fixed linear space $L=\P^{N - \tau}$ 
along a common generator $M$ with $M=V(x_{\tau + 1},\ldots, x_N)$ and  $L=V(x_{N-\tau+1},\ldots, x_N)$. Then: 
  \begin{equation}\label{eq:simplified}
  f = \dy \sum_{i=0}^{\tau} x_i C^i(x_{N-\tau+ 1}, ..., x_N) + D(x_{\tau + 1}, ..., x_N),
  \end{equation}

where the $C_i$'s  are linearly independent quadratic forms depending only on the variables $x_{N-\tau + 1}, ..., x_N$ and $D$ is a cubic in the variables 
 $x_{\tau + 1}, ..., x_N$.
 
 On the contrary, for  every cubic hypersurface $X=V(f)\subset\p^N$ with $f$ as in \eqref{eq:simplified} the polar quadric of a general  point of $\P^N$ is tangent to $V(x_{N-\tau+1},\ldots, x_N)$ along the common generator $V(x_{\tau + 1},\ldots, x_N)$. Moreover  such a cubic hypersurface has vanishing hessian and $\codim(Z)=1$.
\end{thm}

\begin{proof} For notational simplicity in the formulas let $m=N-\tau$.
Let $s=(b_0:...:b_{\tau}:0:...:0) \in M $ and let  $q \in L \subset T_sQ_p$, $q = (a_0:...:a_{\tau}:a_{\tau + 1}:...:a_m:0:...:0)$. 
 
The condition  $q \in T_sQ_p$ for every $q\in L$ and for every $s\in M$ is equivalent to  $sQ_pq^t = 0$ for all $q \in L$ and for all 
$s \in \P^{\tau}$, where by abusing notation we identify $Q_p$ with its associated symmetric matrix. 

 Using the canonical form of Proposition \ref{prop:classe1}, 
we obtain $\frac{\partial f}{\partial x_i} = C^{i}$ $i = 0, ..., \tau$ and
 $\frac{\partial f}{\partial x_i} = \dy \sum_{j = 0}^\tau x_j C^{j}_i   + D_i$ $i = \tau + 1, ..., N$. 
Hence $\frac{\partial^2 f}{\partial x_i \partial x_j} = C_j^i$ for $i= 0, ..., \tau$ and 
$j = \tau + 1, ..., m$. In coordinates:
  $$[b_0:...:b_{\tau}:0:...:0] \left[ \begin{array}{ccccccccccc}
  0 & ... & 0 & C_{\tau + 1}^0 & ... & C^0_{m} &  * & ... & *\\
  \vdots & \ddots & \vdots & \vdots & \ddots & \vdots & \vdots & \ddots & \vdots \\
  0 & ... & 0 & C^{\tau}_{\tau + 1} & ... & C^{\tau}_m & * & ... & * \\
  C_{\tau+1}^0 & ... &C_{\tau+1}^\tau & * & ... & * & * & ... & * \\
  \vdots & \ddots & \vdots & \vdots & \ddots & \vdots & \vdots & \ddots & \vdots \\
  C_m^0& ... & C_m^{\tau} & * & ... & * & * & ... & * \\
  * & ... & * & * & ... & * & * & ... & * \\
  \vdots & \ddots & \vdots & \vdots & \ddots & \vdots & \vdots & \ddots & \vdots \\
  * & ... & * & * & ... & * & * & ... & *
    \end{array} \right] \left[ \begin{array}{c}
    a_0 \\ \vdots \\a_{\tau} \\ a_{\tau + 1} \\ \vdots \\a_m \\ 0 \\ \vdots \\ 0
    \end{array} \right] = 0  $$
  After the first product we get
  $$[0:...:0:\dy \sum_{i=0}^{\tau}b_iC^i_{\tau + 1}:...:\dy \sum_{i=0}^{\tau}b_iC^i_{m}:*:...:*] \left[ \begin{array}{c}
    a_0 \\ \vdots \\a_{\tau} \\ a_{\tau + 1} \\ \vdots \\a_m \\ 0 \\ \vdots \\ 0
    \end{array} \right] = 0.$$
  Then,
  $$\dy \sum_{\small {\begin{array} {c} 0 \leq i \leq \tau \\ \tau + 1 \leq j \leq m \end{array}}} a_j b_i C_j^i = 0.$$
  Hence $C_j^i = 0 $ for $0 \leq i \leq \tau$ and for $\tau + 1 \leq j \leq m$ because $a_i, b_j \in \mathbb{K}$ are arbitrary. So the quadratic forms $C^i$ that, {\it a priori}, are forms in the variables 
$x_{\tau + 1}, ..., x_N$, actually depend only on the variables $x_{m+1}, ..., x_N$ and the result follows since $C_0,\ldots, C_\tau$ are linearly independent being the partial derivatives of $f$ with respect to the corresponding variable.

On the contrary by reversing the argument every cubic hypersurface with an equation of the form \eqref{eq:simplified} is such that the polar quadric of a general  point of $\P^N$ is tangent  to $L$ along the common generator $M$. Moreover $\codim(Z)=1$ by Lemma \ref{tangencia}.
\end{proof}

The next result is crucial to obtain a simplified canonical form for Special Perazzo Cubic Hypersurfaces. 
Our proof was completely inspired by the calculations made by Perazzo in 
\cite[\S 1.14-1.16 p. 344-349]{Pe}, where as always we shall suppose $\codim(Z)=1$, that is $p=0$ in Perazzo's notation.

\begin{thm}\label{analise_sistema_linear} Let $X=V(f) \subset \P^N$ be an irreducible cubic hypersurface with vanishing 
 hessian, not a cone, and such that  $\codim(Z)=1$. Then the following conditions are equivalent:

\begin{enumerate}
\item[i)] $X\subset\p^N$ is a Special Perazzo Cubic Hypersurface with $\dim(Z^*)+1=\sigma$  and
with $L=\dy \bigcap_{w\in W{\rm general}}\p^{N-\dim(Z^*)}_w=\p^{N-\sigma}$;
\medskip

\item[ii)] $\dim(Z^*)=\sigma-1$ and the polar quadric $Q_p$ of a general point  $p\in\P^N$ is tangent to the linear space  
$L=\P^{N - \sigma}$ along  a common generator  $M=\P^{\sigma}$  of   
$Q_p$ with  $\dy \bigcap_{m \in M\setminus\Vert(Q_p)} T_mQ_p = L$.
  \end{enumerate}

Moreover $X\subset\p^N$ is projectively equivalent to
\begin{equation}\label{canSPCH}
V(\sum_{i=0}^{\sigma}x_iC^i(x_{N-\sigma + 1}, ..., x_N) + D(x_{\sigma+ 1}, ..., x_N))\subset\p^N.
\end{equation}  
  
\end{thm}
\begin{proof}
 Let  $X$ be a Special Perazzo Cubic Hypersurface and
let  
$$<Z^*>=M=V(x_{\sigma+1},\ldots, x_N)\subset L=V(x_{N-\sigma+1},\ldots, x_N).$$
 Since $M\subset\Sing X$ by Proposition \ref{prop:lema1,2}  we deduce from Proposition
\ref{prop:classe1} that 
$$f = \dy \sum_{i=0}^{\sigma}x_iC^i(x_{\sigma + 1}, ..., x_N) + D(x_{\sigma+ 1}, ..., x_N),$$
with $C^i(x_{\sigma + 1}, ..., x_N)$ a quadratic form for every $i=0,\ldots,\tau$ and with
$D(x_{\tau + 1}, ..., x_N)$ a cubic form.

In our hypothesis \eqref{linearfiber} yields   $\p^{N-\dim(Z^*)}_w=\Sing Q_w$  for $w\in W=Z^*$ general.
Since $Z^*$ is a hypersurface in $M$,  for $m\in M$ general  there exist $w_1,w_2\in Z^*$ general such
that $m\in<w_1,w_2>$. Then $L\subseteq \Sing Q_m$ because $L\subseteq \Sing Q_{w_i}$ for $i=1,2$. Therefore
$L\subseteq\Sing Q_m$ for every $m\in M$.

By the previous expression of $f$ we get $V(C^i)=Q_{p_i}\subset\p^N$ with $p_i$, $i=0,\ldots,\sigma$, the $i$-th fundamental point of $M$.
The condition $L\subseteq \Sing V(C^i)$
is equivalent to the fact that the quadratic forms $C^i$'s  do not depend on the variables $x_{\sigma+1},\ldots,x_{N-\sigma}$
for every $i=0,\ldots,\sigma$.
Hence $f$ has an equation of the form \eqref{eq:simplified} and from the last part of Theorem \ref{forma_canonica_sist_linear} we deduce  that the general quadric $Q_p$
is tangent to $L$ along the generator $M$. 
From $\dy L\subseteq \dy \bigcap_{m \in M\setminus\Vert(Q_p)} T_mQ_p$ and  from  Lemma \ref{tangencia}
we deduce that equality holds, proving ii).

Let us now assume ii) with  
$L =V(x_{N-\sigma+1},\ldots, x_N)$
and  $M=V(x_{\sigma+1},\dots, x_N)$. 
Let $p\in \p^N$ be a general point and let $q\in L_p=<p,L>$ be a general point. We claim that 
$\Sing Q_q=\Sing Q_p$.

Indeed, if $p=[\mathbf v]$ and if $q=[\mathbf u]$, then $\mathbf u=\alpha\mathbf v+\beta\mathbf z$ with $[\mathbf z]\in L$ and $\alpha\neq 0$. 
Let $s=[\mathbf s]=\operatorname{Sing}Q_q \subset M\subset L$. From 
\eqref{eq:simplified}  we infer $L\subseteq\Sing Q_t$ for every $t\in L$, yielding
\begin{equation}\label{parziale}
\dy \sum_{j=0}^N \frac{\partial^2 f}{\partial x_i \partial x_j}(t)\mathbf s_j =0 \mbox { for every } i=0,\dots, N.
\end{equation}
With the previous notation, we have for every $i=0,\ldots, N$:
$$\dy \sum_{j=0}^N \frac{\partial^2 f}{\partial x_i \partial x_j}(\mathbf u)\mathbf s_j  =  \dy 
\alpha\sum_{j=0}^N \frac{\partial^2 f}{\partial x_i \partial x_j}(\mathbf v)\mathbf s_j + \beta \dy 
\sum_{j=0}^N \frac{\partial^2 f}{\partial x_i \partial x_j}(\mathbf z)\mathbf s_j=\alpha\sum_{j=0}^N \frac{\partial^2 f}{\partial x_i \partial x_j}(\mathbf v)\mathbf s_j$$
where the last equality follows from \eqref{parziale}. In conclusion $\operatorname{Sing}Q_p = \operatorname{Sing}Q_q$ for a general $q\in L_p$, as claimed.

Thus  $\p^{N-\sigma+1}=L_p\subseteq \overline{\mathcal P_X^{-1}(\mathcal P_X(p))}=\p^{N-\dim(Z^*)}$ so that
equality holds since $\dim(Z^*)+1=\sigma$ by hypothesis, concluding the proof of i).
\end{proof}

The next result appears in \cite[\S 1.17,1.18 p. 348, 349, 350]{Pe}, with a small imprecision, because the 
determinantal variety found there is not necessarily irreducible, see also Example \ref{exdet} below.  

\begin{thm}\label{Z dual determinantal} Let notation be as above and 
let $X=V(f) \subset \P^N$ be a Special Perazzo Cubic Hypersurface with $f$ as in \eqref{canSPCH}.
   Then $Z^*$ 
is an irreducible component of the determinantal hypersurface $V(\det(A), x_{\sigma+1},\ldots,x_N)) \subset M=\P^{\dim(Z^*)+1}$
with $A$ a $\sigma\times\sigma$ matrix of linear forms in the variables $x_0,\ldots,x_\sigma$.
In particular  $\deg(Z^*)\leq\sigma=\dim(Z^*)+1$
\end{thm}
\begin{proof}
Let notation be as above, let $f$ be as in \eqref{canSPCH}  and let $p=(y_0:\dots:y_N)=[\mathbf y]\in \p^N$ be a general point. The point  $\operatorname{Sing}Q_p \subset M$ has equations
 $$\dy \sum_{j=0}^N \frac{\partial^2 f}{\partial x_i \partial x_j} (\mathbf y) x_j =0,$$
 $i=0,\ldots, N$,
 and $$Z^*=\dy \overline{\bigcup_{p\in\p^N{\rm general}}\Sing Q_p}$$
 are  the (implicit) parametric equations of $Z^*$ depending on $\mathbf y$. Since $Z^*\subset M=V(x_{\sigma+1},\ldots, x_N)\subset\p^N$, taking into account \eqref{canSPCH}, these equations
 reduce to
 \begin{equation}\label{implicit1}
 \sum_{j=0}^\sigma \frac{\partial C^j}{\partial x_i} (y_{N-\sigma+1},\ldots, y_N) x_j=0=x_{\sigma+1}=\ldots=x_N,
 \end{equation}
 $i=N-\sigma+1,\ldots, N$. Moreover by hypothesis the $\sigma\times(\sigma+1)$ matrix 
 $$C(y_{N-\sigma+1},\ldots,  y_N)=\left[\frac{\partial C^j}{\partial x_i} (y_{N-\sigma+1},\ldots,y_N)\right]$$
 has rank $\sigma$ for $(y_{N-\sigma+1},\ldots,  y_N)$ general because $\Sing Q_p$ is a point for
 $p\in\p^N$ general. 
 
 Let $C_{k,l}^j=\frac{\partial^2 C^j}{\partial x_k\partial x_l}$. Then 
 $$\frac{\partial C^j}{\partial x_i} (y_{N-\sigma+1},\ldots, y_N)=\sum_{k=N-\sigma+1}^NC^j_{k,i} y_k$$
 and \eqref{implicit1} can be written in the form
 $$
 \sum_{j=0}^\sigma (\sum_{k=N-\sigma+1}^NC^j_{k,i} y_k) x_j=0=x_{\sigma+1}=\ldots=x_N,
 $$
 that is 
 \begin{equation}\label{implicit2}
 \sum_{k=N-\sigma+1}^N (\sum_{j=0}^\sigma C^j_{k,i} x_j)y_k=0=x_{\sigma+1}=\ldots=x_N.
 \end{equation}
 Let 
 $$\dy A=A(x_0,\ldots,x_\sigma)=\left[\sum_{j=0}^\sigma C^j_{k,i} x_j\right]$$
 with $k,i\in\{N-\sigma+1,\ldots, N\}$. Then $A$ is a $\sigma\times\sigma$ matrix of linear forms in the variables
 $x_0,\ldots, x_\sigma$ and $Z^*\subseteq V(\det(A), x_{\sigma+1},\ldots, x_N)\subset M$.  Indeed, for $r\in Z^*$ general, every $[\mathbf y]\in\Sing Q_r$ is a solution of \eqref{implicit2}. Thus  \eqref{implicit2} has non trivial solutions for $r\in Z^*$ general, yielding $\det(A(r))=0$ and $Z^*\subseteq V(\det(A),  x_{\sigma+1},\ldots,x_N))$.
 \end{proof}

The next example is a  Special Perazzo Cubic Hypersurface  $X\subset \P^6$ with $\codim(Z)=1$, $\dim(Z^*)=2$ and such that $Z^* \subset \P^3$ is a quadric hypersurface which is an irreducible component of the cubic surface $V(\det(A), x_{\sigma+1},\ldots,x_N))$. Thus the inclusion $$Z^*\subseteq V(\det(A),  x_{\sigma+1},\ldots,x_N))$$ proved in Theorem \ref{Z dual determinantal} can be strict in contrast to
the claim made by Perazzo in \cite[pg. 350]{Pe} according to which $Z^*$ should be  always of degree $\sigma$ for Special Perazzo Cubic Hypersurfaces.

\begin{ex}\label{exdet} { \rm
Let $X = V(f) \subset \P^6$ be the cubic hypersurface given by:
  $$f = x_0x_4x_5 + x_1x_4^2 + x_2x_4x_6 + x_3x_5x_6$$
We have $f_0 = x_4x_5$, $f_1 = x_4^2$, $f_2 = x_4x_6$ e $f_3 = x_5x_6$ 
so that
  $$f_0f_2=f_1f_3.$$
Here we use the same notation as in  Theorem \ref{Z dual determinantal}. Thus
$M=\P^3 =V(x_4, x_5,x_6)$,
For $i=4,5,6$ consider the system:
$$\dy \sum_{j=0}^{3} \frac{\partial^2 f}{\partial x_i \partial x_j}(p)x_j=0=x_4=x_5=x_6,$$
with $p=(y_0:\ldots:y_6)$.

We get
$$\left\{ \begin{array}{ccccccc}
    2x_1y_4 &+ & x_0y_5 & + &  x_2y_6 & = & 0 \\
    x_0p_4 & + & 0 y_5 & + &  x_3y_6 & = & 0 \\
    x_2y_4 & + & x_3y_5 & + &  0y_6 &= & 0
\end{array} \right.$$
Then
 $$\operatorname A= \left[ \begin{array}
  {ccc}
    2x_1 & x_0 &  x_2 \\
    x_0 &  0 &  x_3 \\
    x_2&  x_3 &  0 \end{array}\right]$$
    and $\det(A)= 2x_0x_2x_3 - 2x_1x_3^2 = 2x_3(x_0x_2-x_1x_3).$
Since $Z^*$ is irreducible and non linear we conclude that $$Z^*=V(x_0x_2-x_1x_3, x_4, x_5, x_6) \subset M\subset \p^6.$$
}
 \end{ex}
\medskip

\section{Cubics with Vanishing Hessian in $\P^N$ with $N\leq 6$}\label{Nleq6}

The geometric structure and the canonical forms of cubic hypersurfaces with vanishing hessian and with $\dim(Z^*)=1$
is completely described  in  the  next result.

\begin{thm}\label{mu = 1}
  Let $X \subset \P^N$ be a cubic hypersurface with vanishing hessian, not a cone, with   $\dim(Z^*) = 1$. 
Then $N\geq 4$, $X$ is a   Special Perazzo Cubic Hypersurface and we have:
\begin{enumerate}
 \item[1)]  $\codim(Z)=1$;
\item[2)]  $SZ^* = \P^2\subseteq \Sing X$
\item[3)] $Z^*$ is a conic; 
\item[4)] If $N=4$, then 
$X$ is projectively equivalent to $S(1,2)^*;$
\item[5)] $X$ is projectively equivalent to
$V(x_0x_{N-1}^2 + 2x_1x_{N-1}x_{N} + x_2x_{N}^2 + D(x_3,x_4,...,x_N)).$
\end{enumerate}
\end{thm}
\begin{proof} We can suppose $N\geq 4$ since cubic surfaces with vanishing hessian are easily seen to be cones. The dual of a projective curve, not a line, is a hypersurface so that by Reflexivity we get $\codim(Z)=1$.
Thus $X$ is a Special Perazzo Cubic Hypersurface since the fibers of the Perazzo map
are hyperplanes and hence  $<Z^*>\subseteq\Sing X$.
 
Since  a pencil of quadrics generated
by two  quadrics of rank one contains only two quadrics of rank one we deduce $\deg(Z^*)=2$ (this can be also deduced 
from Theorem \ref{Z dual determinantal}) and $<Z^*>=\p^2$. 

If $N=4$, then $X$ is a scroll in $\p^2$ tangent to $Z^*$ and $X^*$ is  a surface scroll with line directrix $<Z^*>^*$ over the dual conic of $Z^*$ in its linear span. Therefore $X^*$ is projectively equivalent to the cubic rational normal scroll $S(1,2)\subset\p^4$, proving 4).

Part 5) is a particular case of \eqref{canSPCH} taking into account that $C_0, C_1, C_2$ is a basis of $\mathbb K[x_0,x_1,x_2]_2$ because $f_i=C_i$, $i=0,\dots,2$ are linearly independent. Therefore  modulo a change of coordinates in $\P^N$ acting on $V(x_{N-1},x_N)$ as the identity map, we can take $C_0=x_{N-1}^2$, $C_1=2x_{N-1}x_N$ and $C_2=x_N^2$. Then  $Z=V(2z_0z_2-z_1^2)\subset\p^{N*}$ and
$Z^*=V(2x_0x_2-x_1^2, x_3,\ldots, x_N)\subset\p^N$.
\end{proof}
\medskip

As recalled above for  $N\leq 3$ a cubic hypersurface with vanishing hessian is easily seen to be a cone.
We shall now consider  the cases $N=4, 5, 6$.
\medskip

\subsection{Cubics with Vanishing Hessian in $\P^4$}

The next Theorem is essentially a compilation of classical results, more or less known nowadays,
see \cite{Pe},  \cite{CRS} and \cite{GR}. 

\begin{thm}\label{P4}
 Let $X \subset \P^4$ be an irreducible cubic hypersurface with vanishing hessian,  not a cone. Then
\begin{enumerate}
\item[i)] $(\Sing X)_{\operatorname{red}} = \P^2$;
\medskip
\item[ii)] $X^* \simeq S(1,2) \subset \P^4$;
\medskip 
\item[iii)] $X$ is projectively equivalent to a linear external projection of $\operatorname{Seg}(1,2) \subset \P^5$;
\medskip
\item[iv)] $X$ is projectively equivalent to
$V(x_0x_3^2 + 2x_1x_3x_4 + x_2x_4^2)\subset\p^4.$
\end{enumerate}

 \end{thm}

\begin{proof} First of all let us remark that $\codim(Z)=1$. On the contrary by cutting $X$ with a general hyperplane $H$
and by projecting $Z$ from the general point $[H]$ we would obtain a cubic hypersurface in $\p^3$ with vanishing
hessian, which would be a cone. This would imply that $X$ is a cone. Thus $\codim(Z)=1$ and from \eqref{Perazzoestimate} we get
$\dim(Z^*)=1$.  All conclusions now follow from Theorem \ref{mu = 1}, see also Example \ref{Segreproj}, except for iv).
Theorem \ref{mu = 1} implies that  $X$ is projectively equivalent to $Y=V(x_0x_3^2 + 2x_1x_3x_4 + x_2x_4^2+D(x_3,x_4))\subset\p^4$
with $D(x_3,x_4)=(ax_3+bx_4)x_3^2+(cx_3+dx_4)x_4^2$. The projective transformation $x'_0=x_0+ax_3+bx_4$,
$x'_1=x_1$, $x'_2=x_2+cx_3+dx_4$, $x'_3=x_3$ and $x'_4=x_4$ sends $Y$ into 
$V(x_0x_3^2 + 2x_1x_3x_4 + x_2x_4^2)$, as claimed.
\end{proof}
\medskip

\subsection{Cubics with vanishing hessian in $\P^5$}

\begin{thm}\label{P5}
  Let $X = V(f) \subset \P^5$ be a cubic hypersurface with vanishing hessian, not a cone. 
Then :
\begin{enumerate}
\item[i)] $X$ is a Special Perazzo Cubic Hypersurface such that $Z^*$ is a conic so that
$X$ is a scroll in $\p^2$ tangent to the conic $Z^*$;
\medskip

\item[ii)] $X$ is projectively equivalent to $V(x_0x_{4}^2 + 2x_1x_{4}x_5 + x_2x_5^2 +D(x_3,x_4,x_5))\subset \p^5$ with $D(x_3,x_4, x_5)$ a cubic form.
\end{enumerate}
\end{thm}
\begin{proof} Since the polar image of every cubic hypersurface with vanishing hessian in $\p^4$ is a quadric hypersurface,
reasoning as in the proof of Theorem \ref{P4}, we deduce $\codim(Z)=1$. Therefore $\dim(Z^*)\leq 2$ by \eqref{Perazzoestimate}.
If $\dim(Z^*)=2$, then either $X$ is a Special Perazzo Cubic Hypersurface or  two general fibers of the Perazzo map span $\p^5$. In both cases
we would deduce $<Z^*>\subset\Sing(X)$ and $X$ would be a cone by Proposition \ref{prop:classe1}. Thus $\dim(Z^*)=1$ and we can apply Theorem \ref{mu = 1} to conclude.
\end{proof}

\subsection{Cubics with vanishing hessian in $\P^6$}

Let us remark that  since the polar image of every cubic hypersurface in $\p^5$ with vanishing hessian, not a cone, is a quadric hypersurface we deduce that $\codim(Z)=1$ for
each cubic hypersurface in $\p^6$ with vanishing hessian. Therefore \eqref{Perazzoestimate} yields $\dim(Z^*)\leq 2$.

\subsubsection{The case $\dim(Z^*)=1$} 

\begin{thm}\label{P6_1}
 Let $X \subset \P^6$ be an irreducible  cubic hypersurface with vanishing hessian, not a cone, and such that $\dim(Z^*)=1$.
 Then  $X$ is a Special Perazzo Cubic Hypersurface projectively equivalent to
 $V( x_0x_5^2 + 2x_1x_5x_6 + x_2x_6^2 + D(x_3,x_4,x_5,x_6),$
  where $D$ is a cubic form in the variables $x_3$, $x_4$, $x_5$, and $x_6$.

Furthermore the surface $V(D, x_0,x_1,x_2) \subset V(x_0, x_1, x_2 )$ 
is not singular along the line $T=V(x_0,x_1,x_2,x_5,x_6)$.
\end{thm}
\begin{proof} All the conclusions follow from Theorem \ref{mu = 1} except the condition of non singularity of $V(D)$
along $T$ which is necessary for $X$ not to be a cone as one verifies by computing the derivatives of the general 
canonical form.
\end{proof}
\medskip

\begin{rmk}\label{Y6_1}
  {\rm
Using the  canonical form in Theorem \ref{mu = 1}   a direct computation shows that there are only three possibilities for $(\operatorname{Sing}X)_{\red}$ in Theorem \ref{P6_1}.

If $T \cap \operatorname{Sing} V(D) = \emptyset$, then $(\Sing X)_{\red} = \P^2$. If 
$T \cap \operatorname{Sing}V(D) \neq \emptyset$, then either the intersection consists of two points
and $(\Sing X)_{\red} = \P^3_1 \cup \P^3_2$; or it consists of a point and $(\Sing X)_{\red} = \P^3$.}
\end{rmk}

By analyzing only the canonical form Perazzo did not distinguish the three previous cases. The next result describes  the different geometry of $X$ occurring  in the three distinct 
cases. 

\begin{thm}\label{P6_1_1}
  Let $X \subset \P^6$ be a cubic hypersurface having vanishing hessian, not a cone,  with
 $\dim(Z^*) = 1$. Then $X $ is a Special Perazzo Cubic Hypersurface which is a scroll in $\p^2$ tangent to the conic $Z^*$.

 Moreover if $<Z^*>\subsetneq (\Sing X)_{\red}$
 then one of the following conditions holds:
  \begin{enumerate}
    
    \item[i)] If $(\Sing X)_{\red} = \P^3$, then $X$ is also a  scroll in $\p^3$ tangent to the conic $Z^*$.
 \medskip
    
 \item[ii)] If $(\Sing X)_{\red} = \P^3_1 \cup \P^3_2$, then there exist two different structures of scroll
 in  $\P^3$ tangent to $Z^*$ such that the members of the two structures  passing through a general
point  $x \in X$ intersect in a $\P^2_x$ which is tangent to $Z^*$, yielding the original structure of scroll
in $\p^2$ tangent to $Z^*$.
  \end{enumerate}
\end{thm}
\begin{proof}  By Theorem \ref{mu = 1} we know that  $X$ has
a structure of scroll in $\p^2$ tangent to the conic $Z^*$.

By Remark \ref{Y6_1} if $<Z^*>\subsetneq(\Sing X)_{\red}$, then either $(\Sing X)_{\red} = \P^3$
or $(\Sing X)_{\red} = \P^3_1 \cup \P^3_2$. In both cases
the $\P^4$'s passing through a $\P^3$ contained in $\Sing X$ induce on $X$ a structure of scroll in $\p^3$ tangent to
$Z^*$. Moreover   in case ii) the $\p^3$'s of each scroll structure passing through a general $x\in X$ are tangent
to $Z^*$ at the same point $\mathcal P_X(x)$ so that their intersection $\p^2_x$ is also tangent to $Z^*$ at $\mathcal P_X(x)$. By construction this structure coincides with the one described at the beginning of the proof.
\end{proof}

\subsubsection{The case $\dim(Z^*)=2$}

\begin{thm}\label{P6_2}
  Let $X\subset \P^6$ be a cubic hypersurface with vanishing hessian, not a cone,  with $\dim(Z^*)=2$. Then
 \begin{enumerate}
\item[i)]  $X$ is a Special Perazzo Cubic Hypersurface which  is a scroll in $\P^3$ tangent to the surface $Z^*$, which is either a quadric
  or a cubic surface in $\p^3$;
\item[ii)] $X$ is projectively equivalent to 
$$V(\sum_{i=0}^3x_iC^i(x_4,x_5,x_6)+D(x_4,x_5,x_6)),$$ 
where $C^0$, $C^1$, $C^2$, and $C^3$ are quadratic forms  and $D$ is a cubic form;
\item[iii)] $X$ is the dual variety
of a scroll in $\P^2$ over a surface $\widetilde Z\subset \p^3$ which is projectively equivalent to the dual of $Z^*$ in its linear span.
Moreover $X^*$ cuts a general generator of $Z$, which is a cone over $\widetilde Z$  with vertex $<Z^*>^*=\p^2$, along a
plane of the ruling.
\end{enumerate}
\end{thm}
\begin{proof} 
Let $\P^4_{r_i}=\overline{\mathcal P_X^{-1}(r_i)}$, $i=1,2$, be two general fibers of the Perazzo map of $X\subset\p^6$. There are two possibilities:
either $<\P^4_{r_1},\p^4_{r_2}>=\p^5$ (or equivalently $\P^4_{r_1}\cap\p^4_{r_2}=\p^3$)
or $<\P^4_{r_1},\p^4_{r_2}>=\p^6$ (or equivalently $\P^4_{r_1}\cap\p^4_{r_2}=\p^2$). 

In the first case $X$ is a Special Perazzo Cubic Hypersurface by Lemma \ref{sistema especial de Perazzo}.
In the second case Proposition \ref{geramPN} implies $<Z^*>=\p^3\subseteq\Sing X$ but this case 
cannot exist 
by Proposition \ref{prop:lema1,2}. Theorem \ref{Z dual determinantal} yields $\deg(Z^*)\leq 3$, concluding the proof
of i) while part ii) follows from Theorem \ref{analise_sistema_linear}.
\end{proof}

\section{Examples  in higher dimensions}\label{N7}
The classification of cubic hypersurfaces 
with vanishing hessian $X\subset\p^N$ for $N\leq 6$ showed that they all
have  $\codim(Z)=1$  and that they are  Special Perazzo Cubic Hypersurfaces.

In \cite[\S 2]{CRS} two  series of examples of hypersurfaces with vanishing hessian $X\subset\p^N$, $N\geq 5$, of sufficiently high  degree  and such that $\codim(Z)>1$ have been constructed, see \cite[Proposition 2.15, Remark 2.16, Remark 2.19]{CRS}.

We now construct  
some examples of cubic hypersurfaces in $\p^N$, $N\geq 7$, with  $\codim(Z)>1$. Later on we shall provide also
examples of cubic hypersurfaces with vanishing hessian not a cone in $\p^N$ with $N=7, 13, 25$
which are not  Special Perazzo Cubic Hypersurfaces.
\medskip

First we give two different constructions of cubic hypersurfaces with vanishing hessian, not cones,  from 
examples in lower dimension dubbed  {\it concatenation} and {\it juxtaposition}. 
By repeated use of these methods one can prove that for a fixed $\alpha\in \mathbb N$
there exists $N_0=N_0(\alpha)$ such that for every $N\geq N_0$ there exists  a cubic hypersurface $X\subset\p^N$
with vanishing hessian, not a cone, and such that $\codim(Z)=\alpha$. 
\medskip

\begin{ex}{\rm Given  $g=x_0x_3^2+x_1x_3x_4+x_2x_4^2$ let  $\tilde{g} = x_2x_4^2+x_5x_4x_7+x_6x_7^2$. 
We define the concatenation of $g$ and $\tilde{g}$ as the form $f$ obtained by summing up the monomials without 
repetition, that is  
$$f=x_0x_3^2+x_1x_3x_4+x_2x_4^2+x_5x_4x_7+x_6x_7^2.$$

Then $V(f)\subset\p^7$ is a cubic hypersurface with vanishing hessian,
not a cone. Indeed, letting $f_i=\frac{\partial f}{\partial x_i}$ one easily verifies
that the $f_i$'s are linearly independent so that $X$ is not a cone and $Z$ is non degenerate.
Moreover
$$f_0f_2=f_1^2\;\;\;\mbox{ and }\;\;\;f_2f_6=f_5^2,$$
 yielding $\codim(Z)\geq 2$ and $\deg(Z)\geq 3$. Thus $\codim(Z)=2$ since the polar
image of every cubic hypersurface with vanishing hessian in $\p^5$, not a cone, is a quadric hypersurface. }
\end{ex}
\medskip

\begin{ex}{\rm Let $g=x_0x_3^2+x_1x_3x_4+x_2x_4^2$ and let 
$$f=x_0x_3^2+x_1x_3x_4+x_2x_4^2+x_5x_8^2+x_6x_7x_8+x_7x_9^2$$
be the cubic form obtained by the juxtaposition of $g$ with itself, that is $f$ is the sum of $g$ and 
$\tilde{g} = x_5x_8^2+x_6x_7x_8+x_7x_9^2$ obtained by shifting the indexes of $g$.

As above it is possible to verify  immediately that $X=V(f)\subset\p^9$ is not a cone and that the following algebraic
relations hold $f_0f_2=f_1^2,$ $f_5f_7=f_6^2.$ 
}
\end{ex}
\medskip

\begin{ex}{\rm Let $g=x_0x_3^2+x_1x_3x_4+x_2x_4^2$ and let 
$$f=x_0x_3^2+x_1x_3x_4+x_2x_4^2+x_5x_8^2+x_6x_7x_8+x_7x_9^2+x_{1-}x_{13}^2+x_{11}x_{13}x_{14}+x_{12}x_{14}^2$$
be the cubic form obtained by the juxtaposition of $g$ with itself two times. 

Then $X=V(f)\subset\p^{12}$ is not a cone and the following algebraic relations hold: $f_0f_2=f_1^2,$ $f_5f_7=f_6^2$,
$f_{10}f_{12}=f_{11}^2.$ 
}
\end{ex}
\medskip

We conclude with an   example of a  cubic hypersurface $X\subset\p^7$ with vanishing hessian, not a cone, which
is not a Special Perazzo Cubic Hypersurface. 
\medskip

\begin{ex}{\rm Let
\begin{equation}\label{p2p2}
g=\det
\left(
\begin{array}{ccc}
x_0&x_1&x_2\\
x_3&x_4&x_5\\
x_6&x_7&x_8
\end{array}
\right).
\end{equation}
Then $R=V(g)\subset\p^8$ is a cubic hypersurface. If we identify $\p^8$ with $\p(\mathbb M_{3\times 3}(\mathbb K))$, then $R$ is the locus
of matrices of rank at most two and it is naturally identified with the secant variety of $W=\p^2\times\p^2\subset\p^8$,
the locus of matrices of rank 1. Moreover $\Sing R=W$ as schemes.

The polar map $\Phi_R:\p^8\dasharrow\p^8$ is a birational involution sending a matrix $p=[A]$ to its cofactor matrix. Since the cofactor matrix of a rank two 
matrix has rank one and since $\Phi_{R|R}=\mathcal G_R$, we deduce that $\overline{\Phi_R(R)}=\overline{\mathcal G_R(R)}=R^*$ is naturally identified with $W$. By the previous description $\Phi_R$ is an isomorphism on $\p^8\setminus R$ and the closure
of every positive dimensional fiber of $\Phi_R$ (and hence of every fiber of $\mathcal G_R$)  is a $\p^3$. Indeed by homogeneity it is sufficient to verify
this for $\overline{\Phi_R^{-1}(q)}$ with $q=(0:0:0:0:0:0:0:0:1).$
The  $3\times 3$ matrices mapped by $\Phi_R$ to $q$ are exactly the rank  two matrices $X=[x_{i,j}]$ having 
$x_2=x_3=x_5=x_6=x_7=x_8=0$,
i.e. $\overline{\Phi_R^{-1}(q)}=V(x_2,x_5,x_6, x_7,x_8)$  is the closure of the orbit of 
\begin{equation}\label{pR}
p=
\left(
\begin{array}{ccc}
1&0&0\\
0&1&0\\
0&0&0
\end{array}
\right)
\end{equation}
under the natural action.

Let 
$$f=\det
\left(
\begin{array}{ccc}
x_0&x_1&x_2\\
x_3&x_4&x_5\\
x_6&x_7&0
\end{array}
\right).$$

Then $x_8=0$ is the equation of $T_pR$ with $p\in R$ the  point defined in \eqref{pR}.

Let
$$X=R\cap T_pR=V(f)\subset\p^7=V(x_8)\subset\p^8.$$
Then $X\subset\p^7$ is a cubic hypersurface with vanishing hessian, not a cone. Indeed the partial derivatives $f_i$'s are linearly
independent and we have
$$f_0f_4=x_2x_5x_6x_7=f_1f_3.$$
More precisely $Z=\overline{\Phi_X(\p^7)}=V(y_0y_4-y_1y_3)\subset(\p^7)^*,$ is
a rank four quadric with vertex $V=V(y_0,y_1,y_3,y_4)$. Thus
$$Z^*=V(x_0x_4-x_1x_3, x_2,x_5,x_6,x_7)\subset <Z^*>=V(x_2,x_5,x_6,x_7)=\p^3.$$

Hence $<Z^*>=\overline{\Phi_R^{-1}(q)}$ is the fiber of the Gauss map of $R$ passing through $p$, that is the contact locus
on $R$ of the hyperplane $T_pR$, yielding $<Z^*>\subseteq\Sing(X)$ (a fact which can be verified also directly). The variety $Z^*$ is  thus the locus of secant and tangent lines to $W$
passing through $p\in R=SW$. 

The hypersurface $X$ is singular also along $\Sing R\cap T_pR=W\cap T_pR$.
We claim that $(\Sing X)_{\red}=<Z^*>\cup Y_1\cup Y_2$, where each $Y_i\subset<Y_i>=\p^5$
is a Segre 3-fold $\p^1\times\p^2$ and where $<Y_1>\cap <Y_2>=<Z^*>$. Indeed $T_pR\cap W$ is a hyperplane
section of $W$ so that it has degree 6. Moreover  $T_pR\cap W$  contains the following two Segre 3-folds: 
$$\rk\left(
\begin{array}{ccc}
x_0&x_1&x_2\\
x_3&x_4&x_5\\
0&0&0
\end{array}
\right)=1
$$
lying  in $V(x_6,x_7)=\p^5\subset\p^7$
and
$$\rk\left(
\begin{array}{ccc}
x_0&x_1&0\\
x_3&x_4&0\\
x_6&x_7&0
\end{array}
\right)=1
$$
lying in $V(x_2,x_5)=\p^5\subset\p^7$. 

Let $\mathcal P_X:\p^7\dasharrow Z^*$ be the Perazzo map of $X$. Thus for $w\in W=Z^*$ general we have
 $\overline{\mathcal P_X^{-1}(w)}=\p^5_w$. More precisely  if 
$w \in Z^*$, $w=(a_0:a_1:0:a_3:a_4:0:0:0)$ with $a_0a_4-a_1a_3=0$, then:
$$Q_w = \left( \begin{array}{cccccccc}
  0&0&0&0&0&0&0&0 \\
  0&0&0&0&0&0&0&0 \\
  0&0&0&0&0&0& -a_4 & a_3 \\
  0&0&0&0&0&0&0&0 \\
  0&0&0&0&0&0&0&0 \\
  0&0&0&0&0&0& a_1 & -a_0 \\
  0&0& -a_4 &0&0& a_1 &0&0 \\
  0&0& a_3 &0&0& -a_0 &0&0 \\
  \end{array} \right) $$
This matrix has rank two due to the relation $a_0a_4-a_1a_3=0$ and $\Sing Q_w=V(-a_4x_6+a_3x_7, -a_4x_2+a_1x_5)$ 
for $w$ general.
Therefore two general fibers of the Perazzo map intersect in a $\P^3$ and $X=V(f)\subset\p^7$ is
not a Special Perazzo Cubic Hypersurface.

Let us remark that  $X^*\subset Z\subset (\p^7)^*$ is a 4-fold
which by duality is  the projection of $R^*$ from the point $\mathcal G_R(p)=q=(0:0:\ldots:0:1)$.
One can alternatively  deduce   $\dim(X^*)=4$ by  first observing that $T_pR$ cuts a general fiber of the Gauss map of $R$ in a $\p^2$,
which  becomes the general fiber of the Gauss map of $X$ so that  $\dim(X^*)=\dim(X)-2=4$.

To see the vanishing of the hessian of $X$ geometrically
one  remarks that $\Phi_R(T_pR)=Q\subset\p^8$ is a quadric singular at $\Phi_R(p)$. Indeed,  the restriction
of $\Phi_R$ to $T_pR$ is birational onto the image, which is a quadric hypersurface since $\Phi_R$ is given by quadratic equations and it is an involution;  moreover the general positive
dimensional fiber of the restriction of $\Phi_R$ to $T_pR$ is two dimensional while the fiber of the Gauss map
through $p$ has dimension three. Thus projecting $Q$ from $\Phi_R(p)$ one obtains a quadric hypersurface
$Q\subset\p^7$ containing $\Phi_X(\p^7)$. Thus $X\subset\p^7$ has vanishing hessian and it is not difficult
to prove that $Z_X=Q$, also by direct computations as seen above. }
\end{ex}

\begin{rmk}\label{final}{\rm 
One can construct in a similar way  examples of cubic hypersurfaces with vanishing hessian, not cones, such that $\codim(Z)=1$
and $\codim(X^*,Z)>1$. Indeed by taking as $R$ the secant variety to one of the two Severi varieties $W=\mathbb G(1,5)\subset\p^{14}$,
respectively $W=E_6\subset\p^{26}$, and by considering their section by a tangent hyperplane
one gets examples of cubic hypersurfaces $X\subset\p^N$ with $N=13$, respectively  $N=25$, such that $Z_X$ is a quadric hypersurface
while $X^*$ has dimension 8, respectively 16. Thus in the first case $\codim(X^*,Z)=4$
while in the second case $\codim(X^*,Z)=8$.

Let us recall that $\dim(Z)=\rk\Hess_X-1$ while it is not difficult to prove that $\dim(X^*)=\rk_f\Hess_X-2$, where
$\rk_f\Hess_X$ denotes the rank of the matrix $\Hess_X$ modulo the ideal $(f)$.

For the cubic hypersurfaces $X=V(f)=SW\subset\p^{\frac{3n}{2}+1}$, $n=4, 8, 16$, deduced from  the  corresponding Severi varieties $W^n\subset\p^{\frac{3n}{2}+2}$, we have  $\rk \Hess_X=\frac{3n}{2}+1$ and $\rk_f\Hess_X=n+2$ since $\codim(X^*, Z)=\frac{n}{2}$.
Thus for $n=4, 8, 16$ we have $\rk\Hess_X>\rk_f\Hess_X$, something which is somehow unexpected and which can occur
only for non Special Cubic Perazzo Hypersurfaces.

In general it seems quite difficult to construct examples of (cubic) hypersurfaces with vanishing hessian, not cones, such that $\codim(X^*, Z)$
can be arbitrarily large. 
We plan to come back to this intriguing problem elsewhere.
}
\end{rmk}
\medskip

\section*{Acknowledgements}
Our collaboration  began several years ago in Recife while we were trying
to understand, revisit and generalize the ideas contained in 
\cite{Pe}. The first author was  supported by CAPES and FACEPE while the second author  by CNPq, by
the University of Catania and by PRIN "Geometria delle Variet\`a Algebriche".
We would like to thank  Ciro Ciliberto, Thiago Fassarella, Riccardo Re and Aron Simis for many interesting discussions on these topics along the years.

We are deeply indebted to the referee who read very carefully  the whole paper providing  many important remarks and suggestions which 
allowed  us to improve significantly the exposition and to eliminate a lot of imprecisions, both mathematical and lexical. 
Due to his(/her) contributions our original rough manuscript turned into a mathematical paper.

\end{document}